\pdfoutput=1 

\documentclass[a4paper,twoside]{article}

\usepackage{amsmath}       %
\usepackage{amssymb}       %
\usepackage{amsthm}        %
\usepackage{mathscinet}    %
\usepackage{mathrsfs}      
\usepackage{microtype}     
\usepackage{tabularx}      
\usepackage[T1]{fontenc}   
\usepackage{lmodern}       
\usepackage[utf8]{inputenc}
\usepackage{enumitem}      

\usepackage{hyperref}

\newtheoremstyle{citing}
  {3pt}
  {3pt}
  {\itshape}
  {}
  {\bfseries}
  {.}
  {.5em}
  {\thmnote{#3}}
\theoremstyle{citing}
\newtheorem*{citing}{}

\swapnumbers

\theoremstyle{plain}
\newtheorem{theorem}{Theorem}[section]
\newtheorem{lemma}[theorem]{Lemma}
\newtheorem{corollary}[theorem]{Corollary}

\theoremstyle{remark}
\newtheorem{remark}[theorem]{Remark}
\newtheorem{example}[theorem]{Example}
\newtheorem*{question}{Questions}

\theoremstyle{definition}
\newtheorem*{intro_definition}{Definition}
\newtheorem{definition}[theorem]{Definition}
\newtheorem{miniremark}[theorem]{}

\newcounter{counter1}
\newcounter{counter2}

\newcommand{\Var}{\mathbf{V}}     
\newcommand{\IVar}{\mathbf{IV}}   

\newcommand{\Lp}[1]{\mathbf{L}_{#1}}

\newcommand{\rel}{\mathbf{R}}

\newcommand{\grass}[2]{\mathbf{G}(#1,#2)}

\newcommand{\Mypp}{\mathbf{p}}
\newcommand{\Myqq}{\mathbf{q}}

\newcommand{\oball}[2]{\mathbf{U}(#1,#2)}

\newcommand{\cball}[2]{\mathbf{B}(#1,#2)}

\newcommand{\density}{\boldsymbol{\Theta}}

\newcommand{\unitmeasure}[1]{\boldsymbol{\alpha}(#1)}

\newcommand{\id}[1]{\mathbf{1}_{#1}}

\newcommand{\ud}{\,\mathrm{d}}

\DeclareMathOperator{\with}{:}

\newcommand{\project}[1]{#1_\natural}

\newcommand{\perpproject}[1]{#1_\natural^\perp}

\newcommand{\lIm}{[}
\newcommand{\rIm}{]}

\newcommand{\vdim}{{m}}
\newcommand{\codim}{{n-m}}
\newcommand{\adim}{{n}}

\newcommand{\intertextenum}[1]{\setcounter{counter2}{\value{enumi}}\end{enumerate}#1\begin{enumerate}\setcounter{enumi}{\value{counter2}}}

\newcommand{\printRoman}[1]{\setcounter{counter1}{#1}\Roman{counter1}}

\newcommand{\tint}[2]{{\textstyle\int_{#1}^{#2}}}

\newcommand{\tsum}[2]{{\textstyle\sum_{#1}^{#2}}}

\newcommand{\measureball}[2]{{#1}\,{#2}}

\DeclareMathOperator{\without}{\sim}
\DeclareMathOperator{\restrict}{\llcorner}
\DeclareMathOperator{\Clos}{Clos}   
\DeclareMathOperator{\Tan}{Tan}     
\DeclareMathOperator{\spt}{spt}     
\DeclareMathOperator{\im}{im}       
\DeclareMathOperator{\Lip}{Lip}     
\DeclareMathOperator{\grad}{grad}   
\DeclareMathOperator{\dmn}{dmn}     
\DeclareMathOperator{\dist}{dist}   
\DeclareMathOperator{\Hom}{Hom}     
\DeclareMathOperator{\Der}{D}       
\DeclareMathOperator{\pt}{pt}       
\DeclareMathOperator{\ap}{ap}       

\def\polhk#1{\setbox0=\hbox{#1}{\ooalign{\hidewidth
\lower1.5ex\hbox{`}\hidewidth\crcr\unhbox0}}}

\begin{document}


\title{Pointwise differentiability of higher order for sets}
\author{Ulrich Menne}
\date{\today}
\maketitle
\begin{abstract}
	The present paper develops two concepts of pointwise differentiability
	of higher order for arbitrary subsets of Euclidean space defined by
	comparing their distance functions to those of smooth submanifolds.
	Results include that differentials are Borel functions, higher order
	rectifiability of the set of differentiability points, and a
	Rademacher result.  One concept is characterised by a limit procedure
	involving inhomogeneously dilated sets.

	The original motivation to formulate the concepts stems from studying
	the support of stationary integral varifolds.  In particular, strong
	pointwise differentiability of every positive integer order is shown
	at almost all points of the intersection of the support with a given
	plane.
\end{abstract}
\paragraph{MSC-classes 2010}
	51M05 (Primary); 26B05, 49Q20 (Secondary).



\section{Introduction}

\emph{Suppose throughout the introduction that $k$ and $\adim$ are positive
integers, $0 \leq \alpha \leq 1$, $\gamma = k$ if $\alpha = 0$, and $\gamma =
(k,\alpha)$ if $\alpha > 0$.} Before outlining the characterisations and the
differentiability theory for pointwise differentiability of higher order for
sets in Subsections~\ref{subsect:sets_def} and~\ref{subsect:sets_theory}, the
notation is introduced in Subsection \ref{subsect:notation} and the simpler
case of functions is reviewed in Subsection~\ref{subsect:functions}.  The
results on varifolds are summarised in Subsection~\ref{subsect:varifolds}.

\subsection{Notation} \label{subsect:notation}

The notation is consistent with Federer \cite[pp.~669--671]{MR41:1976}.
Therefore the term ``class~$\gamma$'' is used instead of the more common term
``class~$\mathscr{C}^\gamma$'', see \ref{def:k_alpha} and
\ref{def:submanifold}.

Additionally, the following definitions are made.  If $x \in \rel^\adim$
and $A \subset \rel^\adim$, then
\begin{equation*}
	\dist (x,A) = \inf \{ |x-a| \with a \in A \}
\end{equation*}
denotes the \emph{distance of $x$ to $A$}.  Suppose $\vdim$ is an integer with
$0 \leq \vdim \leq \adim$.  If $S$ is an $\vdim$~dimensional subspace
of~$\rel^\adim$, then $\project S$ denotes the orthogonal projection
of~$\rel^\adim$ onto~$S$ with $\project S | S = \id S$ and $S^\perp = \ker
\project S$ denotes the \emph{orthogonal complement of $S$ in $\rel^\adim$}.%
\begin{footnote}
	{Whenever $A$ is a set $\id A$ denotes the identity map of $A$, see
	\cite[p.~669]{MR41:1976}.}
\end{footnote}
The Grassmann manifold $\grass \adim \vdim$ of $\vdim$~dimensional subspaces
of~$\rel^\adim$ is topologised by its injection into $\Hom ( \rel^\adim,
\rel^\adim )$ which maps $S \in \grass \adim \vdim$ onto $\project S$.%
\begin{footnote}
	{Equivalently, the topology on $\grass \adim \vdim$ is characterised
	by the requirement that $\grass \adim \vdim$ becomes a homogeneous
	space through the canonical transitive left action of the orthogonal
	group $\mathbf{O}(n)$ on $\grass \adim \vdim$, see
	\cite[\hyperlink{2_7_1}{2.7.1},
	\hyperlink{3_2_28}{3.2.28}\,(2)\,(4)]{MR41:1976}.}
\end{footnote}
Moreover, the image of a set $A$ under a relation $f$ is denoted by
\begin{equation*}
	f \lIm A \rIm = \{ y \with \text{$(x,y) \in f$ for some $x \in A$} \},
\end{equation*}
see Kelley \cite[p.~8]{MR0370454}.  In this regard all relations, in
particular functions, are considered as subsets of $\dmn f \times \im f$, the
product of their domain and image.

Finally, concerning varifolds, the notation is consistent with Allard
\cite{MR0307015}.

\subsection{Higher order differentiability theory for functions}
\label{subsect:functions}

\begin{intro_definition} [classical, see
	\ref{miniremark:k_jet_for_class_gamma} and \ref{def:pt_diff_fct}]

	Suppose $\vdim$ is an integer, $0 < \vdim < \adim$, $f : \rel^\vdim
	\to \rel^\codim$, and $a \in \rel^\vdim$.

	Then $f$ is termed \emph{pointwise differentiable of order~$\gamma$ at
	$a$} if and only if there exists a polynomial function $P : \rel^\vdim
	\to \rel^\codim$ of degree at most $k$ such that%
	\begin{footnote}
		{The symbol $\cball ar$ denotes the closed ball with centre $a$
		and radius $r$, see \cite[2.8.1]{MR41:1976}.}
	\end{footnote}
	\begin{gather*}
		\lim_{r \to 0+} r^{-k} \sup \{ | f(x)-P(x) | \with x \in
		\cball ar \} = 0 \quad \text{if $\alpha = 0$}, \\
		\limsup_{r \to 0+} r^{-k-\alpha} \sup \{ | f(x)-P(x) | \with x
		\in \cball ar \} < \infty \quad \text{if $\alpha > 0$}.
	\end{gather*}
	In this case $P$ is unique and the \emph{pointwise differentials of
	order $i$ of $f$ at $a$} are defined by
	\begin{equation*}
		\pt \Der^i f (a) = \Der^i P (a) \quad \text{for $i = 0,
		\ldots, k$}.
	\end{equation*}
\end{intro_definition}

These differentials are also called ``$i$th Peano derivatives'', see Zibman
\cite{MR508884}.

The case $(k,\alpha) = (1,0)$ corresponds to classical differentiability.  If
$\codim = 1$ and $f$ is convex, then $f$ is pointwise differentiable of
order~$2$ at $\mathscr{L}^\vdim$ almost all~$a$ by Alexandrov's theorem, see
for instance \cite[Theorem~6.9]{MR3409135}.  Pointwise differentiability of
order~$2$ also plays an important role in the study of viscosity solutions to
nonlinear elliptic equations, see for instance Trudinger \cite{MR995142} and
Caffarelli, Crandall, Kocan, and {\'S}wi{\polhk{e}}ch \cite{MR1376656}.
Examples of arbitrary order of differentiability may be obtained from
Rešetnjak's differentiability result for Sobolev functions, see
\cite{MR0225159-english},%
\begin{footnote}
	{The Russian original is \cite{MR0225159}.}
\end{footnote}
in conjunction with embedding theorems into continuous functions (see O'Neil
\cite{MR0146673} and Stein \cite{MR607898} for related sharp results).

The present development of a higher order differentiability theory of sets
aims at generalising the following theorem concerning functions.  The latter
is readily deduced from known results and included here for expository
reasons.

{ \hypertarget{A}{}
\begin{citing} [Theorem~A, see \ref{thm:pt_diff_functions}]
	Suppose $\vdim$ is an integer, $0 < \vdim < \adim$, $f : \rel^\vdim
	\to \rel^\codim$, and $X$ is the set of $a \in \rel^\vdim$ at which
	$f$ is pointwise differentiable of order~$\gamma$ at $a$.

	Then the following four statements hold.
	\begin{enumerate}
		\item \label{item:intro_thm:borel} The functions $\pt \Der^i
		f$ are Borel functions whose domains are Borel subsets of
		$\rel^\vdim$ for $i = 0, \ldots, k$ and $X$ is a Borel subset
		of $\rel^\vdim$.
		\item \label{item:intro_thm:rect} There exists a sequence of
		functions $g_j : \rel^\vdim \to \rel^\codim$ of class~$\gamma$
		such that $\mathscr{L}^\vdim \big ( X \without
		\bigcup_{j=1}^\infty \{ x \with f(x) = g_j(x) \} \big ) = 0$.%
		\begin{footnote}
			{The symbol $\mathscr{L}^\vdim$ denotes the $\vdim$
			dimensional Lebesgue measure, see
			\cite[2.6.5]{MR41:1976}.}
		\end{footnote}
		\item \label{item:intro_thm:diff} If $g : \rel^\vdim \to
		\rel^\codim$ is of class~$\gamma$ and $Y = \{ y \with
		f(y)=g(y) \}$, then
		\begin{gather*}
			\pt \Der^i f (a) = \Der^i g(a) \quad \text{for $i = 0,
			\ldots, k$}, \\
			\lim_{r \to 0+} r^{-k-\alpha} \sup \{ | f(x)-g(x)|
			\with x \in \cball ar \} = 0
		\end{gather*}
		at $\mathscr{L}^\vdim$ almost all $a \in X \cap Y$.
		\item \label{item:intro_thm:rademacher} If $\alpha = 1$, then
		$f$ is pointwise differentiable of order~$k+1$ at
		$\mathscr{L}^\vdim$ almost all $a \in X$.
	\end{enumerate}
\end{citing} }

For differentiability in Lebesgue spaces $\Lp p ( \mathscr{L}^\vdim,
\rel^\codim )$ with $1 < p \leq \infty$ similar results to Theorem~\hyperlink
AA were developed by Calder{\'o}n and Zygmund, see \cite[Theorems 5, 9, 10,
13]{MR0136849}.  The present proof of Theorem \hyperlink AA mainly relies on a
characterisation of almost everywhere approximate differentiability of
order~$\gamma$, see Isakov \cite{MR897693-english},%
\begin{footnote}
	{The Russian original is \cite{MR897693}.}
\end{footnote}
and some techniques from \cite{MR41:1976}.  A more detailed description of its
proof will be given jointly with that of Theorem~\hyperlink BB for sets below.

\subsection{Defining and characterising higher order pointwise
differentiability for sets} \label{subsect:sets_def}

For sets the first concept of pointwise differentiability is defined as
follows.

\begin{intro_definition} [see \ref{def:diff_sets},
	\ref{remark:same_tangent_cone}, \ref{miniremark:cone}, and
	\ref{thm:eq_diff_sets}]

	Suppose $A \subset \rel^\adim$ and $a \in \rel^\adim$.

	Then $A$ is termed \emph{strongly pointwise differentiable of
	order~$\gamma$ at $a$} if and only if there exist an integer $\vdim$
	with $0 \leq \vdim \leq \adim$, $S \in \grass \adim \vdim$, and a
	polynomial function $P : S \to S^\perp$ of degree at most $k$ such
	that $B = \{ \chi + P(\chi) \with \chi \in S \}$ satisfies $a \in B$
	and
	\begin{gather*}
		\lim_{r \to 0+} r^{-k} \sup \{ | \dist (x,A) - \dist (x,B) |
		\with x \in \cball ar \} = 0 \quad \text{if $\alpha = 0$}, \\
		\limsup_{r \to 0+} r^{-k-\alpha} \sup \{ | \dist (x,A) - \dist
		(x,B) | \with x \in \cball ar \} < \infty \quad \text{if
		$\alpha > 0$}.
	\end{gather*}
\end{intro_definition}

In this case $a$ belongs to the closure of $A$, $\Tan(A,a) = \Tan (B,a)$,%
\begin{footnote}
	{The tangent cone $\Tan (A,a)$ consists of all $v \in \rel^\adim$ such
	that for $\varepsilon > 0$ there exist $x \in A$ and $0 < r < \infty$
	such that $|x-a| < \varepsilon$ and $|r(x-a)-v| < \varepsilon$, see
	\cite[3.1.21]{MR41:1976}.  In set-valued analysis this cone is called
	``contingent cone'' of $A$ at $a$, see \cite[4.1.1]{MR2458436}.}
\end{footnote}
$\vdim$ is determined by $A$ and $a$, a plane $S \in \grass \adim \vdim$ is
admissible in the definition if and only if $S^\perp \cap \Tan (A,a) = \{ 0
\}$, and $P$ is determined by $k$, $A$, and $(a,S)$, see
\ref{remark:same_tangent_cone}, \ref{miniremark:cone}, and
\ref{thm:eq_diff_sets}.  Unlike in the case of functions, $\vdim$ may depend
on~$a$, and $\project S (a)$ need not to belong to the interior of $\project S
\lIm A \rIm$ relative to~$S$.

All these remarks also hold with respect to the following weaker
differentiability requirement which treats the sets $A$ and $B$ in an
asymmetric manner.

\begin{intro_definition} [see \ref{def:diff_sets} and
	\ref{def:definition_pt_differential_sets}]

	Suppose $A \subset \rel^\adim$ and $a \in \rel^\adim$.

	Then $A$ is termed \emph{pointwise differentiable of order~$\gamma$ at
	$a$} if and only if there exist an integer $\vdim$ with $0 \leq \vdim
	\leq \adim$, $S \in \grass \adim \vdim$, and a polynomial function $P
	: S \to S^\perp$ of degree at most $k$ such that $B = \{ \chi +
	P(\chi) \with \chi \in S \}$ satisfies $a \in B$ and
	\begin{gather*}
		\lim_{r \to 0+} r^{-1} \sup \{ | \dist (x,A) - \dist (x,B) |
		\with x \in \cball ar \} = 0, \\
		\lim_{r \to 0+} r^{-k} \sup \{ \dist (x,B) \with x \in A \cap
		\cball ar \} = 0 \quad \text{if $\alpha = 0$}, \\
		\limsup_{r \to 0+} r^{-k-\alpha} \sup \{ \dist (x,B) \with x
		\in A \cap \cball ar \} < \infty \quad \text{if $\alpha > 0$}.
	\end{gather*}
	The \emph{pointwise differential $\pt \Der^k A$ of order~$k$ of~$A$}
	is the function whose domain is the set of $(a,S)$ such that these
	conditions are satisfied for some $\vdim$ and $P$ with $\alpha = 0$
	and whose value at such $(a,S)$ equals $\Der^k ( P \circ \project S )
	(a)$.
\end{intro_definition}

Requiring $S = \Tan (A,a)$, one could define an equivalent notion of pointwise
differentials of higher order whose domains are subsets of $\rel^\adim$, see
\ref{remark:uniqueness_derivative_sets}.  The present definition is chosen
for notational effectiveness.

The motivation for the asymmetric definition is that in case $\alpha = 0$ it
may be characterised by an inductive procedure considering pointwise limits of
the distance functions associated to inhomogeneously dilated sets, suitably
subtracting homogeneous polynomial functions of smaller degree at each step,
see \ref{thm:inductive_blowup}.  In particular, if $A$ is pointwise
differentiable of order $1$ at $a$ and $S = \Tan (A,a)$, then pointwise
differentiability of order $2$ of $A$ at $a$ is equivalent to the requirement
that for some homogeneous polynomial function $Q : S \to S^\perp$ of degree
$2$, the sets $A_s = \big \{ s^{-1} \project S (\chi-a) + s^{-2} \perpproject
S (\chi-a) \with \chi \in A \big \}$ and $B = \{ \chi + Q ( \chi ) \with \chi
\in S \}$ satisfy
\begin{equation*}
	\lim_{s \to 0+} \dist (x,A_s) = \dist (x,B) \quad \text{for $x \in
	\rel^\adim$},
\end{equation*}
see \ref{remark:equiv_convergence_sets} for alternate formulations of the last
condition.  Therefore pointwise differentiability of order~$2$ corresponds to
``twice differentiability'' as defined by the author in
\cite[p.~2253]{snulmenn.mfo1230}.  Similar inhomogeneous dilations of
order~$2$ with weak convergence of the correspondingly restricted Hausdorff
measures to ``approximate tangent paraboloids'' occur in Anzellotti and
Serapioni \cite[\S3]{MR1285779}.

Basic examples of sets which are $\mathscr{H}^\vdim$ almost everywhere
strongly pointwise differentiable of order~$2$ are relative boundaries of
$\vdim+1$ dimensional convex subsets of $\rel^\adim$, see
\ref{example:alexandrov}.  However, the main motivation of the author to study
higher order differentiability theory of sets is given by sets for which no
local graphical representation is available.

In order to characterise the preceding two concepts, one may assume
\begin{align*}
	& A \subset \big \{ \chi \with | \perpproject S (\chi-a) | \leq \kappa |
	\project S (\chi-a) | \big \} \\
	& \qquad \text{for some integer $\vdim$ with $0 \leq \vdim \leq
	\adim$, $S \in \grass \adim \vdim$, and $0 \leq \kappa < \infty$}
\end{align*}
by \ref{miniremark:cone}.  In this case both pointwise differentiability and
strongly pointwise differentiability are characterised in terms of vertical
closeness of order~$\gamma$ of~$A$ near~$a$ to a polynomial function $P : S
\to S^\perp$ of degree at most $k$ and the behaviour of the set $\project S
\lIm A \rIm$ near $\project S(a)$, see \ref{thm:eq_diff_sets}.  The difference
of the two concepts is solely given by the degree to which $\project S \lIm A
\rIm$ is required to cover~$A$ near~$\project S(a)$.  This characterisation
provides the basic link between the differentiability theory of sets and that
of functions on which the further development rests.

It would also appear natural to investigate whether the regularity conditions,
pointwise and strong pointwise differentiability of higher order, can be
characterised in terms of the behaviour of the various higher order tangent
sets of set-valued analysis designed to capture higher order tangential
behaviour of possibly irregular sets, see \cite[\S\,4.7]{MR2458436}.

\subsection{Higher order differentiability theory for sets}
\label{subsect:sets_theory}

The next theorem summarises the main results of the present paper on
differentiability theory of sets.  Their formulation is completely analogous
to Theorem~\hyperlink{A}{A} but their proofs are more complex.

{ \hypertarget B{}
\begin{citing} [Theorem~B, see
\ref{thm:diff_order_one}, \ref{thm:derivative_Borel}, and
\ref{thm:rademacher_sets}]
	Suppose $\vdim$ is an integer, $0 \leq \vdim \leq \adim$, $A \subset
	\rel^\adim$, $X$ is the set of $a \in \rel^\adim$ such that $A$ is
	pointwise [strongly pointwise] differentiable of order~$\gamma$ at~$a$
	with $\dim \Tan (A,a) = \vdim$, $Y$ is set of $a \in \rel^\adim$ such
	that $A$ is pointwise differentiable of order~$1$ at~$a$ with
	$\dim \Tan(A,a) = \vdim$, and $\tau : Y \to \grass \adim \vdim$
	satisfies $\tau (a) = \Tan (A,a)$ for $a \in Y$.

	Then the following four statements hold.
	\begin{enumerate}
		\item \label{intro:rademacher_sets:borel} The function $\tau$
		is a Borel function and its domain~$Y$ is a Borel set, the
		functions $\pt \Der^i A$ are Borel functions whose domains are
		Borel sets for $i = 1, \ldots, k$, and the set $X$ is a Borel
		set.
		\item \label{intro:rademacher_sets:approx} There exists a
		countable collection of $\vdim$~dimensional submanifolds
		of class~$\gamma$ of $\rel^\adim$ covering $\mathscr{H}^\vdim$
		almost all of $X$.%
		\begin{footnote}
			{The symbol $\mathscr{H}^\vdim$ denotes the $\vdim$
			dimensional Hausdorff measure, see
			\cite[2.10.2]{MR41:1976}.}
		\end{footnote}
		\item \label{intro:rademacher_sets:closeness} If $B$ is an
		$\vdim$ dimensional submanifold of class~$\gamma$ of
		$\rel^\adim$, then $\mathscr{H}^\vdim$ almost all $a \in B
		\cap X$ satisfy $\pt \Der^i A (a,\cdot) = \pt \Der^i B
		(a,\cdot)$ for $i = 0, \ldots, k$ and
		\begin{gather*}
			\lim_{r \to 0+} r^{-k-\alpha} \sup \{ \dist ( x, B )
			\with x \in A \cap \cball ar \} = 0. \\
			\left [ \lim_{r \to 0+} r^{-k-\alpha} \sup \{ | \dist
			( x, A ) - \dist ( x, B ) | \with x \in \cball ar \} =
			0. \right ]
		\end{gather*}
		\item \label{intro:rademacher_sets:points} If $\alpha = 1$,
		then $A$ is pointwise [strongly pointwise] differentiable of
		order~$k+1$ at $\mathscr{H}^\vdim$ almost all $a \in X$.
	\end{enumerate}
\end{citing} }

In order to prove \hyperlink AA\,\eqref{item:intro_thm:borel} and~\hyperlink
BB\,\eqref{intro:rademacher_sets:borel}, the main task is to prove that the
functions in question are Borel \emph{subsets} of suitable complete, separable
metric spaces.  By a classical result in descriptive set theory, see
\ref{thm:borel_maps}, this then readily yields
\eqref{intro:rademacher_sets:borel}.  This pattern of proof is taken from
\cite[3.1.1]{MR41:1976}.  As the natural domain of the polynomial functions
associated with the differentials of sets, the tangent plane, depends on the
point considered, some additional considerations, see \ref{lemma:convergence}
and \ref{lemma:varying_poly_fcts}, are needed to prove closedness of the
auxiliary sets in the case of sets.

The proofs of~\hyperlink AA\,(2) and~\hyperlink
BB\,\eqref{intro:rademacher_sets:approx} rely on
\eqref{intro:rademacher_sets:borel} and a direct consequence, see
\ref{thm:isakov}, of Isakov's characterisation of functions which are almost
everywhere approximately differentiable of order~$\gamma$, see
\cite{MR897693-english}.  Whereas the case of functions is immediate, the case
of sets requires the construction of an auxiliary function
to which \ref{thm:isakov} can be applied.  The function is constructed
using the countable $\vdim$ rectifiability%
\begin{footnote}
	{A subset of $\rel^\adim$ is called countably $\vdim$ rectifiable if
	and only if it can be covered by the union of a countable family of
	Lipschitzian images of subsets of $\rel^\vdim$, see
	\cite[3.2.14\,(2)]{MR41:1976}.}
\end{footnote}
of $X$, see \ref{thm:diff_order_one}.  It inherits the higher order
differentiability properties of $A$ by \ref{lemma:subset_pt_diff_order_gamma}.

The proofs of~\hyperlink AA\,\eqref{item:intro_thm:diff} and~\hyperlink
BB\,\eqref{intro:rademacher_sets:closeness} rest on a special case of a
differentiability theorem for functions on varifolds from Kolasiński and
the author \cite[4.4]{MR3625810}.  The short proof of the presently
required case is included in \ref{thm:big_O_little_o} for the convenience of
the reader.  In the case of functions, one may directly apply
\ref{thm:big_O_little_o} to the modulus of the difference of the two functions
involved.  In the case of sets, one instead applies that theorem to a sequence
of auxiliary functions constructed so as to encode all necessary information
on the relative position of the sets $A$ and $B$.

Finally, as in Liu \cite[1.6]{MR2414427}, \hyperlink
AA\,\eqref{item:intro_thm:rademacher} and \hyperlink
BB\,\eqref{intro:rademacher_sets:points} follow from the two statements
respectively preceding it in conjunction with a Lusin type approximation of
functions of class~$(k,1)$ by functions of class~$k+1$, see
\cite[3.1.15]{MR41:1976}.

\subsection{An application to stationary integral varifolds}
\label{subsect:varifolds}

The following are two important questions for integral varifolds (a concept of
generalised submanifolds with positive integer multiplicity) which either have
``bounded mean curvature and no boundary'' or are stationary.

\begin{question}
	Suppose $0 \leq \kappa < \infty$, $V$ is an at least two-dimensional
	integral varifold in $\rel^\adim$, $\| \delta V \| \leq \kappa \| V
	\|$, and $A = \spt \| V \|$.
	\begin{enumerate}
		\item \label{item:question:bounded_var} How regular needs $A$
		to be near $\mathscr{H}^\vdim$ almost all of its points?
		\item \label{item:question:stationary} How much \emph{more}
		regular needs $A$ to be if $\kappa = 0$, i.e. if $V$ is
		stationary?
	\end{enumerate}
\end{question}

Calling a point of $A$ regular if and only if it possesses a neighbourhood $U$
such that $A \cap U$ is a submanifold of class~$1$ of $\rel^\adim$, it is
known that \eqref{item:question:bounded_var} does not entail regularity
$\mathscr{H}^\vdim$ almost everywhere; in fact, a sequence of increasingly
irregular examples was constructed by Allard \cite[8.1\,(2)]{MR0307015},
Brakke \cite[6.1]{MR485012}, and Kolasiński and the author \cite[10.3,
10.8]{MR3625810}.  On the other hand various weaker properties almost
everywhere resembling the behaviour of closed submanifolds of class~$2$ were
established under the hypotheses of \eqref{item:question:bounded_var} by
Brakke \cite[5.8]{MR485012}, Schätzle \cite[Theorems~4.1, 5.1, 6.1,
6.2]{MR2064971} and \cite[Theorems 3.1, 4.1]{MR2472179}, White
\cite[Theorem~2]{MR2747434}, and the author \cite[4.11]{snulmenn.poincare},
\cite[10.2]{snulmenn.decay}, \cite[4.8, 5.2]{snulmenn.c2},
\cite[14.2]{MR3528825}, and \cite[6.8]{MR3626845}.

The only regularity properties valid across points with higher multiplicity
and specific to the stationary case, see \eqref{item:question:stationary}, are
a number of increasingly delicate maximum principles by Solomon and White
\cite{MR1017330}, Ilmanen \cite[Theorem~A]{MR1402732}, and Wickramasekera
\cite[Theorem~19.1]{MR3171756} and \cite[Theorem~1.1]{MR3268871}.  The present
paper makes the first contribution to \eqref{item:question:stationary} valid
across points with higher multiplicity which is different from a maximum
principle.

{ \hypertarget C{}
\begin{citing} [Theorem~C, see \ref{corollary:approx_implies_pt}]
	Suppose $\vdim$ is an integer, $2 \leq \vdim \leq \adim$, $S \in
	\grass \adim \vdim$, $V$~is an $\vdim$ dimensional stationary integral
	varifold in $\rel^\adim$, and $A = \spt \| V \|$.

	Then $A$ is strongly pointwise differentiable of every positive
	integer order at $\mathscr{H}^\vdim$ almost all $a \in A \cap S$.
\end{citing} }

In \ref{example:hesitate_to_vanish} an example is constructed that shows that
Theorem~\hyperlink CC constitutes a regularity property of integral varifolds
setting the stationary case apart of that of bounded mean curvature and no
boundary, see \ref{remark:better_than_c2}.  As far as pointwise
differentiability is concerned, it relies on ``local maximum estimates'' for
subsolutions to the Laplace equation on varifolds by Michael and Simon
\cite[3.4]{MR0344978} and the general principle that control in an approximate
sense tends to entail control in an integral sense in the presence of an
elliptic partial differential equation.  This principle was first discovered
by Schätzle in \cite[Theorem 3.1]{MR2472179} and was since used by author
in \cite[5.2]{snulmenn.c2} and jointly with Kolasiński in
\cite[9.2]{MR3625810}.  Its present implementation in \ref{lemma:zero}
using Michael and Simon's result is considerably simpler than previous
approaches.  Strong pointwise differentiability then follows using a result of
Kolasiński and the author \cite[10.4]{MR3625810} which relies on
Almgren's multiple valued functions from \cite{MR1777737}.

Theorem~\hyperlink CC naturally raises the following regularity question for
stationary integral varifolds that appears to be more tractable than that of
possible almost everywhere regularity: \emph{Suppose $V$ is an $m$ dimensional
stationary integral varifold in $\mathbf R^n$ and $A = \spt \| V \|$.  Is $A$
necessarily [strongly] pointwise differentiable of every positive integer
order at $\mathscr H^m$~almost all $a \in A$?}

It has been announced by the author in
\cite[Corollary~2\,(1)]{snulmenn.mfo1230} that under weaker hypotheses than
those of \eqref{item:question:bounded_var} the set $A$ is pointwise
differentiable of order~$2$ at $\mathscr{H}^\vdim$ almost all of its points.
Theorem~\hyperlink BB is a significantly generalised version of the first part
of the proof of that result.  The second part of the proof along with some
generalisations of Theorem~\hyperlink CC shall appear elsewhere.  The
formulation of Theorem~\hyperlink BB as a separate result for general subsets
of Euclidean space shall facilitate the use of this technique outside the
varifold context.
 
\subsection{Subsequent developments}

Since the publication of the initial version of this paper, two sequels have
been written.  Firstly, Santilli introduced the corresponding approximate
notion of pointwise differentiability for sets (see~\cite[3.8,
3.19]{IUMJ_7645} and the footnote
to~\ref{def:definition_pt_differential_sets}).  He thus obtains a
characterisation of higher order rectifiability, see~\cite[1.2]{IUMJ_7645};
a problem left open by Anzellotti and Serapioni in~\cite{MR1285779}.
Secondly, the author developed an analogous theory of pointwise
differentiability of higher order for distributions
in~\cite{arXiv:1803.10855v1}.  The latter is related to the afore-mentioned
regularity question for stationary integral varifolds as is elaborated upon
in~\cite[\S\,1.3]{arXiv:1803.10855v1}.

\subsection{Organisation of the paper}

After Section~\ref{sect:preliminaries} on preliminaries,
Section~\ref{sec:diff_sets} provides definitions and characterisations for
higher order differentiability of sets.  Sections~\ref{sect:diff_fcts}
and~\ref{sect:diff_sets} treat the higher order differentiability theory for
functions and sets respectively.  In Section~\ref{sect:example} an example for
use in Section~\ref{sect:varifolds} on varifolds is constructed.  Finally,
Appendix~\ref{sect:appendix} contains a table with brief descriptions of the
items employed from \cite{MR41:1976}.

\subsection{Acknowledgements}

The author would like to thank Mario Santilli for reading part of the
manuscript and for bringing a series of papers of Isakov to his attention,
Dr~Sławomir Kolasiński for helping him to become acquainted with some
of these results available only in Russian, and Dr~Yangqin Fang for pointing
him to \cite{MR1979936}.  The initial version of this paper (see
\href{https://arxiv.org/abs/1603.08587v1}{\path{arXiv:1603.08587v1}}) was
written while the author worked at the Max Planck Institute for Gravitational
Physics (Albert Einstein Institute) and the University of Potsdam.

\section{Preliminaries} \label{sect:preliminaries}

In this section, firstly, an a~priori estimate for polynomial functions, see
\ref{lemma:poly_control}, and a resulting uniqueness theorem, see
\ref{thm:poly_uniqueness}, are proven.  Secondly, classical notions of higher
order differentiability are compiled, see \ref{def:k_alpha},
\ref{def:pt_diff_fct}, and \ref{def:submanifold}.

\begin{lemma} \label{lemma:poly_control}
	Suppose $1 \leq M < \infty$.

	Then there exists a positive, finite number $\Gamma$ with the
	following property.

	If $k$ is a nonnegative integer, $k \leq M$, $S$ is a Hilbert space,
	$\dim S \leq M$, $a \in S$, $0 < r < \infty$, $Y$ is a normed
	vectorspace, $P : S \to Y$ is a polynomial function of degree at
	most~$k$, and $X \subset \cball ar$ satisfies $\dist ( x, X ) \leq
	\Gamma^{-1} r$ for $x \in \cball ar$, then there holds
	\begin{equation*}
		\sup \big \{ r^i \| \Der^i P ( x ) \| \with \text{$x \in
		\cball ar$, $i = 0, \ldots, k$} \big \} \leq \Gamma \sup \{ |
		P (x) | \with x \in X \}.
	\end{equation*}
\end{lemma}

\begin{proof}
	Using translations, homotheties, and the equation
	\begin{equation*}
		|y| = \sup \big \{ | \alpha (y)| \with \text{$\alpha \in
		Y^\ast$ and $\| \alpha \| \leq 1$} \big \} \quad \text{for $y
		\in Y$},
	\end{equation*}
	see \cite[\printRoman 2.3.15]{MR0117523}, it is sufficient to prove
	the assertion resulting from additionally hypothesising $a = 0$, $r =
	1$, and $Y = \rel$ in the body of the lemma.
	
	If the remaining assertion were false for some $M$, there would exist
	a sequence $\Gamma_j$ with $\Gamma_j \to \infty$ as $j \to \infty$ and
	sequences $k_j$, $S_j$, $P_j$, and $X_j$ showing that $\Gamma_j$ does
	not have the property described in the remaining assertion.  One could
	assume that, for some $k$ and $S$, there would hold
	\begin{equation*}
		k_j = k, \quad S_j = S, \quad
		\sup \big \{ \| \Der^i P_j ( x ) \| \with \text{$x \in
		S \cap \cball 01$, $i = 0, \ldots, k$} \big \} = 1
	\end{equation*}
	for every positive integer $j$.  Since the space of polynomial
	functions $P : S \to \rel$ of degree at most $k$ is finite
	dimensional and the function on that space with value
	\begin{equation*}
		\sup \big \{ \| \Der^i P ( x ) \| \with \text{$x \in
		S \cap \cball 01$, $i = 0, \ldots, k$} \big \}
	\end{equation*}
	at $P$ is a norm, see \cite[\hyperlink{1_10_2}{1.10.2},
	\hyperlink{1_10_4}{1.10.4}, \hyperlink{3_1_11}{3.1.11}]{MR41:1976},
	possibly passing to a subsequence, there would exist a polynomial
	function $P : S \to \rel$ such that
	\begin{gather*}
		\sup \big \{ \| \Der^i P ( x ) \| \with \text{$x \in
		S \cap \cball 01$, $i = 0, \ldots, k$} \big \} = 1,
		\\
		\sup \{ |P(x)| \with x \in S \cap \cball 01 \} \leq
		\liminf_{j \to \infty} \sup \{ | P_j(x) | \with x \in X_j \} =
		0.
	\end{gather*}
	This would be a contradiction.
\end{proof}

\begin{remark}
	A conceptually similar lemma appears in Campanato
	\cite[Lemma~2.\printRoman 1]{MR0167862} and is attributed there to
	De~Giorgi.
\end{remark}

\begin{theorem} \label{thm:poly_uniqueness}
	Suppose $k$ is a nonnegative integer, $S$ is a Hilbert space, $\dim S <
	\infty$, $Y$ is a normed vectorspace, $P : S \to Y$ is a polynomial
	function of degree at most $k$, $a \in S$, $X \subset S$, and
	\begin{equation*}
		\lim_{r \to 0+} r^{-1} \sup \dist ( \cdot, X ) \lIm S \cap
		\cball ar \rIm = 0, \quad \lim_{r \to 0+} r^{-k} \sup | P |
		\lIm X \cap \cball ar \rIm = 0.
	\end{equation*}

	Then $P = 0$.
\end{theorem}

\begin{proof}
	Applying \ref{lemma:poly_control} for each sufficiently small $r$
	yields $\Der^i P(a) = 0$ for $i = 0, \ldots, k$, hence $P = 0$ by
	\cite[\hyperlink{3_1_11}{3.1.11}]{MR41:1976}.
\end{proof}

\begin{definition} \label{def:k_alpha}
	Suppose $k$ is a nonnegative integer, $0 < \alpha \leq 1$, $X$ and $Y$
	are normed vectorspaces, and $f$ maps a subset of $X$ into $Y$.

	Then $f$ is said to be \emph{of class $(k,\alpha)$} if and only if $f$
	is of class $k$ and $\Der^k f$ locally satisfies a Hölder
	condition with exponent $\alpha$.%
	\begin{footnote}
		{The map $f$ is called of class $k$ if and only if its domain
		is open and it is $k$ times continuously differentiable, see
		\cite[3.1.11]{MR41:1976}.}
	\end{footnote}
	Moreover, $f$ is called a \emph{diffeomorphism of class~$(k,\alpha)$}
	if and only if $f$ is a homeomorphism, $f$ is of class $(k,\alpha)$,
	and $f^{-1}$ is of class~$(k,\alpha)$.
\end{definition}

\begin{miniremark} \label{miniremark:class_ka}
	Suppose $k$ is a nonnegative integer, $0 < \alpha \leq 1$, $\vdim$ is
	a positive integer, $U$ is an open subset of $\rel^\vdim$, $Y$ is a
	normed vectorspace, and $f_i : U \to \bigodot^i ( \rel^\vdim, Y )$ for
	$i = 0, \ldots, k$.%
	\begin{footnote}
		{If $V$ and $W$ are vectorspaces, then $\bigodot^0 (V,W)=W$
		and $\bigodot^i (V,W)$ is the vectorspace of all symmetric $i$
		linear maps from $V^i$ into $W$ whenever $i$ is a positive
		integer, see \cite[1.10.1]{MR41:1976}.}
	\end{footnote}
	Then \cite[\hyperlink{3_1_11}{3.1.11},
	\hyperlink{3_1_14}{3.1.14}]{MR41:1976} and \ref{lemma:poly_control}
	may be used to verify that $f_0$ is of class $(k,\alpha)$ with $\Der^i
	f_0 = f_i$ for $i = 0, \ldots, k$ if and only if
	\begin{equation*}
		\sup \Big \{ |x-a|^{-k-\alpha} \big | f_0 (x) - \tsum{i=0}k
		\langle (x-a)^i/i!, f_i (a) \rangle \big | \with \text{$a,x \in
		K$, $a \neq x$} \Big \} < \infty
	\end{equation*}
	whenever $K$ is a compact subset of $U$.%
	\begin{footnote}
		{If $V$ and $W$ are vectorspaces, $i$ is a positive
		integer, and $\phi \in \bigodot^i (V,W)$, then
		\begin{equation*}
			\langle v^i/i!, \phi \rangle = i!^{-1} \phi ( v,
			\ldots, v ) \quad \text{for $v \in V$},
		\end{equation*}
		see \cite[1.9.1, 1.10.1, 1.10.4]{MR41:1976}. Similarly,
		$\langle v^i/i!, \phi \rangle = \phi$ if $i = 0$ and $\phi \in
		\bigodot^0 (V,W)$.}
	\end{footnote}
\end{miniremark}

\begin{miniremark} \label{miniremark:k_jet_for_class_gamma}
	Suppose $k$ is a nonnegative integer, $0 \leq \alpha \leq 1$, $\gamma
	= k$ if $\alpha = 0$ and $\gamma = (k,\alpha)$ if $\alpha > 0$, $X$
	and $Y$ are normed vectorspaces, $a \in U \subset X$, $f : U \to Y$ is
	of class~$\gamma$, and $P$ is the $k$~jet of $f$ at $a$.%
	\begin{footnote}
		{The $k$~jet of~$f$ at~$a$ is the polynomial function $P : X
		\to Y$ of degree at most $k$ satisfying $P (x) = \sum_{i=0}^k
		\langle (x-a)^i/i!, \Der^i f(a) \rangle$ for $x \in X$, see
		\cite[3.1.11]{MR41:1976}.}
	\end{footnote}
	Then \cite[\hyperlink{3_1_11}{3.1.11}]{MR41:1976} yields
	\begin{gather*}
		\lim_{x \to a} | f(x)-P(x) |/|x-a|^k = 0 \quad \text{if
		$\alpha = 0$}, \\
		\limsup_{x \to a} | f(x)-P(x) |/|x-a|^{k+\alpha} < \infty
		\quad \text{if $\alpha > 0$}.
	\end{gather*}
\end{miniremark}

\begin{definition} \label{def:pt_diff_fct}
	Suppose $k$ is a nonnegative integer, $0 \leq \alpha \leq 1$, $\gamma
	= k$ if $\alpha = 0$ and $\gamma = (k,\alpha)$ if $\alpha > 0$, $X$
	and $Y$ are normed vectorspaces, $f$~maps a subset of $X$ into $Y$,
	and $a \in X$.

	Then $f$ is called \emph{pointwise differentiable of order~$\gamma$
	at~$a$} if and only if there exist an open subset $U$ of $X$ and a
	function $g : U \to Y$ of class $\gamma$ such that
	\begin{gather*}
		a \in U \subset \dmn f, \quad f(a) = g(a), \\
		\lim_{x \to a} | f(x)-g(x) |/|x-a|^k = 0 \quad \text{if
		$\alpha = 0$}, \\
		\limsup_{x \to a} | f(x)-g(x) |/|x-a|^{k+\alpha} < \infty
		\quad \text{if $\alpha > 0$}.
	\end{gather*}
	Whenever $f$ is pointwise differentiable of order~$k$ at $a$ one
	defines (see \ref{miniremark:k_jet_for_class_gamma}) the
	\emph{pointwise differential of order~$i$ of $f$ at $a$} by
	\begin{equation*}
		\pt \Der^i f(a) = \Der^i g(a) \quad \text{for $i = 0, \ldots,
		k$}.
	\end{equation*}
\end{definition}

\begin{remark} \label{remark:pt_diff_fct_low_order}
	A function is pointwise differentiable of order~$0$ at $a$ if and only
	if it is continuous at $a$. In this case $\pt \Der^0 f(a) = f(a)$.
	Similarly, a function is pointwise differentiable of order~$1$ at $a$
	if and only if it is differentiable at $a$ in which case $\pt \Der^1
	f(a) = \Der f(a)$.
\end{remark}

\begin{remark}
	If $f$ is $k$ times differentiable at $a$, then $f$ is pointwise
	differentiable of order~$k$ at $a$ and $\pt \Der^i f(a)= \Der^i f(a)$
	for $i = 0, \ldots, k$; in fact, one may employ an induction argument
	based on the fact that $\Lip f = \sup \| \Der f \| \lIm U \rIm$
	whenever $U$ is open and convex with $U \subset \dmn \Der f$, see
	\cite[\hyperlink{2_2_7}{2.2.7}, \hyperlink{3_1_1}{3.1.1}]{MR41:1976},
	the one-dimensional case on which appears in Weil
	\cite[p.~589]{MR1363850}.%
	\begin{footnote}
		{If $g$ is a map between metric spaces, then $\Lip g$ is its
		Lipschitz constant, see \cite[2.2.7]{MR41:1976}.}
	\end{footnote}
\end{remark}

\begin{definition} \label{def:submanifold}
	Suppose $k$ and $\adim$ are positive integers, $\vdim$ is an integer,
	$0 \leq \vdim \leq \adim$, $0 < \alpha \leq 1$, and $B \subset
	\rel^\adim$.

	Then $B$ is called an \emph{$\vdim$ dimensional submanifold of class
	$(k,\alpha)$} if and only if for each $b \in B$ there exists a
	neighbourhood $U$ of $b$ in $\rel^\adim$, a diffeomorphism $f : U \to
	\rel^\adim$ of class $(k,\alpha)$, and an $\vdim$ dimensional
	subspace $T$ of $\rel^\adim$ with
	\begin{equation*}
		f \lIm B \cap U \rIm = T \cap \im f.
	\end{equation*}
\end{definition}

\begin{remark} \label{remark:hoelder_ok}
	The basic properties of maps and submanifolds of class $k$ given in
	\cite[\hyperlink{3_1_18}{3.1.18},
	\hyperlink{3_1_19}{3.1.19}]{MR41:1976} remain valid for maps and
	submanifolds of class $(k,\alpha)$.
\end{remark}

\section{Basic characterisations} \label{sec:diff_sets}

In the present section the key definitions concerning differentiability of
higher order for sets are provided in \ref{def:diff_sets} and
\ref{def:definition_pt_differential_sets}.  The main characterisations of
these concepts are proven in \ref{thm:eq_diff_sets} and
\ref{thm:inductive_blowup}; the first of which takes a particularly simple
form if the set is associated to the graph of a function, see
\ref{corollary:cmp_fct_set}.

\begin{miniremark} \label{miniremark:isolated_points}
	Suppose $A \subset \rel^\adim$ and $a \in \rel^\adim$.  Then $\Tan
	(A,a) = \{ 0 \}$ if and only if $a$ is an isolated point of $A$, as
	may be verified using \cite[\hyperlink{3_1_21}{3.1.21}]{MR41:1976}.
\end{miniremark}

\begin{miniremark} \label{miniremark:diff_sets}
	Suppose $A \subset \rel^\adim$, $B \subset \rel^\adim$, $a \in (
	\Clos A ) \cap ( \Clos B )$, and $0 < r < \infty$.%
	\begin{footnote}
		{The closure of a set $A$ is denoted $\Clos A$, see
		\cite[p.~669]{MR41:1976}.}
	\end{footnote}
	Then one verifies%
	\begin{footnote}
		{The symbol $\oball ar$ denotes the open ball with centre $a$
		and radius $r$, see \cite[2.8.1]{MR41:1976}.}
	\end{footnote}
	\begin{gather*}
		\begin{split}
			& \sup | \dist ( \cdot, A ) - \dist ( \cdot, B ) |
			\lIm \oball ar \rIm \\
			& \qquad \leq \sup \big ( \dist ( \cdot, A ) \lIm B
			\cap \oball a{2r} \rIm \cup \dist ( \cdot, B ) \lIm A
			\cap \oball a{2r} \rIm \big ).
		\end{split}
	\end{gather*}
\end{miniremark}

\begin{definition} \label{def:diff_sets}
	Suppose $k$ and $\adim$ are positive integers, $0 \leq \alpha \leq 1$,
	$\gamma = k$ if $\alpha = 0$ and $\gamma = (k,\alpha)$ if $\alpha >
	0$, and $A \subset \rel^\adim$.

	Then $A$ is called \emph{pointwise [strongly pointwise] differentiable
	of order $\gamma$ at $a$} if and only if there exists a submanifold
	$B$ of class $\gamma$ of $\rel^\adim$ such that $a \in B$,
	\begin{equation*}
		\lim_{r \to 0+} r^{-1} \sup | \dist (\cdot, A) - \dist (
		\cdot, B ) | \lIm \cball ar \rIm = 0,
	\end{equation*}
	and
	\begin{gather*}
		\begin{gathered}
			\lim_{r \to 0+} r^{-k} \sup \dist ( \cdot, B ) \lIm A
			\cap \cball ar \rIm = 0 \quad \text{if $\alpha =
			0$}, \\
			\limsup_{r \to 0+} r^{-k-\alpha} \sup \dist ( \cdot,
			B) \lIm A \cap \cball ar \rIm < \infty \quad
			\text{if $\alpha > 0$}.
		\end{gathered} \\
		\left [
		\begin{gathered}
			\lim_{r \to 0+} r^{-k} \sup | \dist ( \cdot, A
			) - \dist ( \cdot, B ) | \lIm \cball ar \rIm
			= 0 \quad \text{if $\alpha = 0$}, \\
			\limsup_{r \to 0+} r^{-k-\alpha} \sup | \dist
			( \cdot, A ) - \dist ( \cdot, B ) | \lIm
			\cball ar \rIm < \infty \quad \text{if
			$\alpha > 0$}.
		\end{gathered}
		\right ]
	\end{gather*}
\end{definition}

\begin{remark} \label{remark:same_tangent_cone}
	It follows that $a \in \Clos A$. Moreover, one verifies $\Tan (A,a) =
	\Tan (B,a)$, in particular $\Tan (A,a)$ is a $\dim B$ dimensional
	subspace of $\rel^\adim$.
\end{remark}

\begin{remark}
	In the bracketed case the condition
	\begin{equation*}
		\lim_{r \to 0+} r^{-1} \sup | \dist (\cdot, A) - \dist (
		\cdot, B ) | \lIm \cball ar \rIm = 0
	\end{equation*}
	is redundant.
\end{remark}

\begin{remark} \label{remark:agreement_strong_normal}
	If $(k,\alpha)=(1,0)$, then pointwise differentiability and strong
	pointwise differentiability agree.
\end{remark}

\begin{remark} \label{remark:diffeo_invariance}
	If $U$ is an open subset of $\rel^\adim$, $g : U \to \rel^\adim$ is a
	diffeomorphism of class~$\gamma$, $A \subset U$, $a \in U$, and $A$ is
	pointwise [strongly pointwise] differentiable of order~$\gamma$ at
	$a$, then $g \lIm A \rIm$ is pointwise [strongly pointwise]
	differentiable of order~$\gamma$ at $g(a)$, as may be verified using
	\ref{miniremark:diff_sets}.
\end{remark}

\begin{remark} \label{remark:diff_sets_zero_dim}
	Suppose $A$, $k$, $\alpha$, and $\gamma$ are as the hypotheses of
	\ref{def:diff_sets}.  Then $A$ is pointwise [strongly pointwise]
	differentiable of order $\gamma$ at $a \in \rel^\adim$ with $\Tan
	(A,a) \in \grass \adim 0$ if and only if $a$ is isolated in $A$ by
	\ref{miniremark:isolated_points}.
\end{remark}

\begin{example} \label{example:subset_of_plane}
	Suppose $\adim$ is a positive integer, $\vdim$ is an integer, $0 \leq
	\vdim \leq \adim$, and $A \subset S \in \grass \adim \vdim$.  Then $A$
	is pointwise differentiable of order~$1$ at $a$ with $\Tan (A,a) = S$
	whenever $a \in S$ and $\density^\vdim ( \mathscr{H}^\vdim \restrict S
	\without \Clos A, a ) = 0$ which is the case at $\mathscr{H}^\vdim$
	almost all $a \in \Clos A$ by
	\cite[\hyperlink{2_10_19}{2.10.19}\,(4)]{MR41:1976}.%
	\begin{footnote}
		{The $\vdim$ dimensional density of a measure $\phi$ over
		$\rel^\adim$ at $a$ equals
		\begin{equation*}
			\density^\vdim ( \phi, a ) = \lim_{r \to 0+}
			\frac{\measureball \phi {\cball ar}}{\unitmeasure
			\vdim r^\vdim},
		\end{equation*}
		where $\unitmeasure \vdim =
		\measureball{\mathscr{L}^\vdim}{\cball 01}$ if $\vdim > 0$ and
		$\unitmeasure 0 = 1$, see \cite[2.7.16\,(1),
		2.10.19]{MR41:1976}.}
	\end{footnote}
\end{example}

\begin{miniremark} \label{miniremark:cone}
	If $\adim$ is a positive integer, $\vdim$ is an integer, $0 \leq \vdim
	\leq \adim$, $A \subset \rel^\adim$, $a \in \Clos A$, $S \in \grass
	\adim \vdim$, and $S^\perp \cap \Tan (A,a) = \{ 0 \}$, then there
	exist $r > 0$ and $0 \leq \kappa < \infty$ such that
	\begin{equation*}
		A \cap \cball ar \subset \big \{ \chi \with | \perpproject S
		(\chi-a) | \leq \kappa | \project S (\chi-a) | \big \}.
	\end{equation*}
\end{miniremark}

\begin{theorem} \label{thm:eq_diff_sets}
	Suppose $k$ and $\adim$ are positive integers, $\vdim$ is an integer,
	$0 \leq \vdim \leq \adim$, $0 \leq \alpha \leq 1$, $\gamma = k$ if
	$\alpha = 0$ and $\gamma = (k,\alpha)$ if $\alpha > 0$, $A \subset
	\rel^\adim$, $a \in \Clos A$, $S \in \grass \adim \vdim$, $0 \leq
	\kappa < \infty$, and
	\begin{equation*}
		A \subset \big \{ \chi \with | \perpproject S (\chi-a) | \leq
		\kappa | \project S (\chi-a) | \big \}.
	\end{equation*}

	Then the following three conditions are equivalent.
	\begin{enumerate}
		\item \label{item:eq_diff_sets:eq} There exists an $\vdim$
		dimensional submanifold $B$ of class $\gamma$ of $\rel^\adim$
		such that $B$ satisfies the conditions of \ref{def:diff_sets}.
		\item \label{item:eq_diff_sets:manifold} The set $\project S
		\lIm A \rIm$ is strongly pointwise differentiable of order $1$
		[order $\gamma$] at $\project S (a)$ with $\Tan ( \project S
		\lIm A \rIm, \project S (a) ) = S$ and there exists an $\vdim$
		dimensional submanifold $B$ of class $\gamma$ of $\rel^\adim$
		such that $a \in B$ and
		\begin{gather*}
			\lim_{r \to 0+} r^{-k} \sup \dist ( \cdot, B) \lIm A
			\cap \cball ar \rIm = 0 \quad \text{if $\alpha = 0$},
			\\
			\limsup_{r \to 0+} r^{-k-\alpha} \sup \dist ( \cdot,
			B) \lIm A \cap \cball ar \rIm < \infty \quad \text{if
			$\alpha > 0$}.
		\end{gather*}
		\item \label{item:eq_diff_sets:polynomial} The set $\project S
		\lIm A \rIm$ is strongly pointwise differentiable of order $1$
		[order $\gamma$] at $\project S (a)$ with $\Tan ( \project S
		\lIm A \rIm, \project S ( a )) = S$ and there exists a
		function $f : S \to S^\perp$ of class~$\gamma$ satisfying
		\begin{equation*}
			\lim_{r \to 0+} r^{-k} \sup \big \{ | \perpproject S
			(\chi)- f ( \project S (\chi) ) | \with \chi \in A
			\cap \project S^{-1} \lIm \cball{ \project S(a)} r
			\rIm \big \} = 0
		\end{equation*}
		if $\alpha = 0$, and
		\begin{equation*}
			\limsup_{r \to 0+} r^{-k-\alpha} \sup \big \{ |
			\perpproject S (\chi)- f ( \project S (\chi) ) | \with
			\chi \in A \cap \project S^{-1} \lIm \cball{ \project
			S(a)} r \rIm \big \} < \infty
		\end{equation*}
		if $\alpha > 0$.
	\intertextenum{In this case the $k$ jet $P$ of $f$ at $\project S (a)$
	is uniquely determined by $k$, $A$, and $(a,S)$ and the following
	three additional statements hold.}
		\item \label{item:eq_diff_sets:add_manifold} An $\vdim$
		dimensional submanifold $B$ of class~$\gamma$ of $\rel^\adim$
		satisfies \eqref{item:eq_diff_sets:eq} if and only if there
		exists a function $f : S \to S^\perp$ of class~$\gamma$ such
		that
		\begin{equation*}
			B \cap U = \{ \chi + f ( \chi ) \with \chi \in S \}
			\cap U \quad \text{for some neighbourhood $U$ of $a$}
		\end{equation*}
		and the $k$~jet of $f$ at $\project S(a)$ equals $P$.
		\item \label{item:eq_diff_sets:add_half_manifold} The
		statement \eqref{item:eq_diff_sets:add_manifold} holds with
		``\eqref{item:eq_diff_sets:eq}'' replaced by
		``\eqref{item:eq_diff_sets:manifold}''.
		\item \label{item:eq_diff_sets:add_fct} A function $f : S \to
		S^\perp$ of class~$\gamma$ satisfies
		\eqref{item:eq_diff_sets:polynomial} if and only if its
		$k$~jet at $\project S(a)$ equals $P$.
	\end{enumerate}
\end{theorem}

\begin{proof}
	If $\vdim = 0$ then $A = \{ a \}$.  Therefore one may assume $\vdim >
	0$ by \ref{remark:diff_sets_zero_dim}.

	Clearly, \eqref{item:eq_diff_sets:polynomial} implies
	\eqref{item:eq_diff_sets:manifold} with $B = \{ \chi + f ( \chi )
	\with \chi \in S \}$.  Moreover, if $B$ satisfies
	\eqref{item:eq_diff_sets:eq} then there exists a function $f : S \to
	S^\perp$ of class~$\gamma$ such that
	\begin{equation*}
		B \cap U = \{ \chi + f ( \chi ) \with \chi \in S \} \cap U
		\quad \text{for some neighbourhood $U$ of $a$}
	\end{equation*}
	by \ref{remark:hoelder_ok}, \ref{remark:same_tangent_cone}, and
	\cite[\hyperlink{3_1_18}{3.1.18},
	\hyperlink{3_1_19}{3.1.19}\,(4)]{MR41:1976}.
	
	Next, two implications will be shown.

	\emph{Firstly, if $B$ satisfies \eqref{item:eq_diff_sets:manifold}
	then $B$ satisfies \eqref{item:eq_diff_sets:eq}.}  In this regard,
	notice that
	\begin{equation*}
		\Tan (A,a) \subset \Tan (B,a), \quad \Tan ( \project S \lIm A
		\rIm, \project S (a) ) \subset \project S \lIm \Tan (A,a)
		\rIm,
	\end{equation*}
	where the second inclusion is implied by $ S^\perp \cap \Tan (A,a) =
	\{ 0 \}$.  It follows $\project S \lIm \Tan (B,a) \rIm = S$, the
	linear map $\project S | \Tan (B,a)$ is univalent%
	\begin{footnote}
		{The term ``univalent'' is also known as ``injective''.}
	\end{footnote}
	and
	\begin{equation*}
		\Tan (A,a) = \Tan (B,a).
	\end{equation*}
	Therefore, in view of \cite[\hyperlink{3_1_18}{3.1.18},
	\hyperlink{3_1_19}{3.1.19}\,(4)]{MR41:1976}, one may assume
	\begin{equation*}
		| \perpproject{S} ( \chi-b ) | \leq \kappa | \project S
		(\chi-b) | \quad \text{whenever $\chi, b \in B$}
	\end{equation*}
	possibly enlarging $\kappa$ and shrinking $B$.  Also observe that
	\ref{remark:same_tangent_cone} with $A$, $a$, and $B$ replaced by
	$\project S \lIm A \rIm$, $\project S (a)$, and the submanifold
	furnished by the strong pointwise differentiability of order $1$
	[order $\gamma$] of $\project S \lIm A \rIm$ at $\project S (a)$,
	\cite[\hyperlink{3_1_19}{3.1.19}\,(5)]{MR41:1976}, and
	\begin{equation*}
		\dist ( \project S ( \chi ), \project S \lIm A \rIm ) \leq
		\dist ( \chi, \project S \lIm A \rIm ) \quad \text{for $\chi
		\in \rel^\adim$}
	\end{equation*}
	imply for $\delta (r) = \sup \dist ( \cdot, \project S \lIm A
	\rIm ) \lIm S \cap \cball{ \project{S} (a)} r \rIm$ that
	\begin{gather*}
		\lim_{r \to 0+} r^{-1} \delta (r) = 0. \\
		\left [ \lim_{r \to 0+} r^{-k} \delta (r) = 0 \quad
		\text{if $\alpha = 0$}, \qquad \limsup_{r \to 0+}
		r^{-k-\alpha} \delta (r) < \infty \quad \text{if $\alpha
		> 0$}. \right ]
	\end{gather*}
	If $\chi \in B$ and $| \chi-a | \leq r$, then, as $\Clos \project S
	\lIm A \rIm = \project S \lIm \Clos A \rIm$, there exist $x \in \Clos
	A$ with $| \project S(x-\chi) | = \dist ( \project S (\chi), \project
	S \lIm A \rIm )$ and $b \in \Clos B$ with $|x-b| = \dist (x,B)$,
	hence, noting
	\begin{gather*}
		| \project {S} (\chi-x) | \leq \delta (r), \quad |
		\project S (x-a) | \leq | \project{S} (x-\chi) | + | \project
		S ( \chi-a) | \leq 2r, \\
		|x-a| \leq ( 1 + \kappa ) | \project{S} (x-a) | \leq 2 ( 1 +
		\kappa ) r, \quad |x-b| \leq d ( ( 3 + 2 \kappa ) r ), \\
		| \project{S} ( \chi-b ) | \leq | \project{S} ( \chi-x ) | + |
		\project S (x-b) | \leq \delta (r) + |x-b|, \\
		| \chi-b | \leq ( 1 + \kappa ) | \project{S} (\chi-b) |,
	\end{gather*}
	where $d(s) = \sup \dist ( \cdot, B ) \lIm A \cap \cball as \rIm$ for
	$0 < s < \infty$, one estimates
	\begin{equation*}
		| \chi-x | \leq | \chi-b | + | b - x | \leq (2+\kappa) \big (
		\delta (r) + d ((3+2\kappa)r) \big ).
	\end{equation*}
	Therefore, in view of \ref{miniremark:diff_sets}, $B$ satisfies
	\eqref{item:eq_diff_sets:eq} and the first implication is proven.

	\emph{Secondly, if $f : S \to S^\perp$ is of class~$\gamma$ and $B =
	\{ \chi + f ( \chi ) \with \chi \in S \}$
	satisfies~\eqref{item:eq_diff_sets:eq}, then $f$ satisfies
	\eqref{item:eq_diff_sets:polynomial}.}  For this purpose choose $r >
	0$ and $0 \leq \lambda < \infty$ such that
	\begin{equation*}
		\Lip ( f | \cball {\project{S}(a)}{(2+\kappa)r} ) \leq
		\lambda.
	\end{equation*}
	If $\chi \in B$ and $| \project S (\chi-a) | \leq s \leq r$, then
	there exists $x \in \Clos A$ with $|\chi-x| = \dist ( \chi,A )$ and,
	noting $\project S (x) \in \Clos \project S \lIm A \rIm$ and $| \chi-a
	| \leq ( 1 + \lambda ) s$, one infers
	\begin{equation*}
		\dist ( \project S ( \chi ), \project S \lIm A \rIm ) \leq |
		\project S (\chi-x) | \leq | \chi-x | \leq \sup \dist ( \cdot,
		A ) \lIm B \cap \cball a {(1+\lambda )s)} \rIm,
	\end{equation*}
	hence, noting $\project S \lIm B \rIm = S$, the set $\project S \lIm A
	\rIm$ is strongly pointwise differentiable of order $1$ [order
	$\gamma$] at $\project S ( a) \in \Clos \project S \lIm A \rIm$ and
	\begin{equation*}
		\Tan ( \project S \lIm A \rIm, \project S (a) ) = S
	\end{equation*}
	by \ref{miniremark:diff_sets} and \ref{remark:same_tangent_cone}.
	Moreover, if $\chi \in A$ and $| \project S ( \chi-a ) | \leq s \leq
	r$, then $| \chi-a| \leq ( 1 + \kappa) s$ and there exists $b \in B$
	such that
	\begin{equation*}
		| \chi-b | = \dist ( \chi, B ) \leq \sup \dist ( \cdot, B )
		\lIm A \cap \cball a{(1+\kappa) s } \rIm \leq ( 1 + \kappa )
		r,
	\end{equation*}
	hence, defining $x = \project S (\chi) +f ( \project S (\chi))$ and
	noting $\project{S} (x) = \project{S} (\chi)$, one obtains
	\begin{gather*}
		| \project S (b-a) | \leq | \project S (\chi-b) | + | \project
		S ( \chi-a ) |  \leq (2+\kappa)r, \\
		x \in B, \quad | x-b | \leq ( 1 + \lambda ) | \project{S}
		(x-b) | \leq ( 1 + \lambda ) | \chi-b |, \\
		| \perpproject{S} ( \chi ) - f ( \project S ( \chi ) ) | = |
		\chi - x | \leq | \chi - b | + | x-b | \leq ( 2 + \lambda ) |
		\chi - b |.
	\end{gather*}
	Therefore $f$ satisfies \eqref{item:eq_diff_sets:polynomial} and the
	second implication is proven.

	In view of the first implication with $A$, $a$, and $B$ replaced by
	$\project S \lIm A \rIm$, $\project S (a)$, and $S$, the uniqueness of
	$P$ and \eqref{item:eq_diff_sets:add_fct} follow from
	\ref{thm:poly_uniqueness} and \ref{miniremark:k_jet_for_class_gamma}.
	Therefore the preceding two implications and the initial paragraph
	yield the equivalence of
	\eqref{item:eq_diff_sets:eq}--\eqref{item:eq_diff_sets:polynomial} and
	\eqref{item:eq_diff_sets:add_fct} implies
	\eqref{item:eq_diff_sets:add_manifold} and
	\eqref{item:eq_diff_sets:add_half_manifold}.
\end{proof}

\begin{definition} \label{def:definition_pt_differential_sets}
	Suppose $\adim$ is a positive integer and $A \subset \rel^\adim$.

	Then for every positive integer $k$ the function $\pt \Der^k A$ is
	defined (see \ref{miniremark:cone} and \ref{thm:eq_diff_sets}) to be
	the function whose domain consists of all $(a,S)$ such that $a \in
	\Clos A$, the set $A$ is pointwise differentiable of order~$k$ at $a$,
	and
	\begin{equation*}
		S \in \grass \adim {\dim \Tan (A,a)}, \quad S^\perp \cap \Tan
		(A,a) = \{ 0 \}
	\end{equation*}
	and whose value at such $(a,S)$ equals the unique $\phi \in \bigodot^k
	(\rel^\adim,\rel^\adim)$ such that for some function $f : S \to
	S^\perp$ of class~$k$ there holds $\phi = \Der^k ( f \circ \project S
	) ( a)$ and
	\begin{gather*}
		\lim_{r \to 0+} r^{-1} \sup | \dist (\cdot,A) - \dist ( \cdot,
		B) | \lIm \cball ar \rIm = 0, \\
		\lim_{r \to 0+} r^{-k} \sup \dist ( \cdot, B) \lIm A \cap
		\cball ar \rIm = 0,
	\end{gather*}
	where $B = \{ \chi + f ( \chi ) \with \chi \in S \}$.  Abbreviate $\pt
	\Der^1 A = \pt \Der A$.  Moreover, the function $\pt \Der^0 A$ is
	defined to be the function whose domain consists of all $(a,S) \in (
	\Clos A ) \times \bigcup_{\vdim=0}^\adim \grass \adim \vdim$ and whose
	value at such $(a,S)$ equals $\perpproject{S} (a)$.  The value $\pt
	\Der^i A(a,S)$ is called the \emph{pointwise differential of order~$i$
	of $A$ at~$(a,S)$} whenever $i$ is a nonnegative integer and $(a,S)
	\in \dmn \pt \Der^i A$.%
	\begin{footnote}
		{Anticipating the results of this paper and its logical
		sequel~\cite{IUMJ_7645}, we remark that -- employing the
		terminology of approximate differentiation from~\cite[3.8,
		3.19]{IUMJ_7645} -- the following proposition may be deduced
		from~\cite[4.1, 4.3, 4.11]{IUMJ_7645}
		and~\ref{miniremark:cone},
		\ref{thm:eq_diff_sets}\,\eqref{item:eq_diff_sets:add_manifold}:
		\emph{Whenever $a \in \mathbf R^n$, $A \subset \mathbf R^n$,
		$k$ is a positive integer, $0 \leq \alpha \leq 1$, $\gamma =
		k$ if $\alpha = 0$ and $\gamma = ( k, \alpha )$ if $\alpha >
		0$, $A$~is approximately differentiable of order~$1$ at $a$,
		$T = \ap \Tan (A,a)$, and $m = \dim T \geq 1$, the following
		two conditions are equivalent:
		\begin{enumerate}
			\item The set $A$ is approximately differentiable of
			order~$(k,\alpha)$ at $a$.
			\item \label{item:santilli:pt} There exists a subset
			$B$ of $\mathbf R^n$ that is pointwise differentiable
			of order $\gamma$ at $a$ and satisfies the conditions
			$\Tan (B,a) = T$ and $\density^\vdim ( \mathscr H^m
			\restrict A\without B, a) = 0$.
		\end{enumerate}
		In this case, $\ap \Der^i A (a) = \pt \Der^i B (a,T)$ for $i =
		2, \ldots, k$.} In~\eqref{item:santilli:pt}, one may require
		$B \subset A$.}
	\end{footnote}
\end{definition}

\begin{remark} \label{remark:uniqueness_derivative_sets}
	If $k$ is a positive integer and $(a,S) \in \dmn \pt \Der^k A$, then,
	by \ref{thm:eq_diff_sets}\,\eqref{item:eq_diff_sets:add_manifold}, the
	values $\pt \Der^i A (a,S)$ for $i = 0, \ldots, k$ determine $\pt
	\Der^i A(a,\cdot)$ for $i = 0, \ldots, k$.
\end{remark}

\begin{corollary} \label{corollary:cmp_fct_set}
	Suppose $k$ and $\adim$ are positive integers, $\vdim$ is an integer,
	$0 \leq \vdim \leq \adim$ $0 \leq \alpha \leq 1$, $\gamma = k$ if
	$\alpha = 0$ and $\gamma = (k,\alpha)$ if $\alpha > 0$, $S \in \grass
	\adim \vdim$, $x \in X \subset S$, $f : X \to S^\perp$, $f$ is
	continuous at $x$, $A = \{ \chi + f(\chi) \with \chi \in \dmn f \}$,
	and $a = x + f(x)$.

	Then the following two conditions are equivalent.
	\begin{enumerate}
		\item \label{item:eq_diff_sets:set2} The set $A$ is pointwise
		[strongly pointwise] differentiable of order $\gamma$ at $a$,
		$\Tan (A,a) \in \grass \adim \vdim$, and $S^\perp \cap \Tan
		(A,a) = \{ 0 \}$.
		\item \label{item:eq_diff_sets:fct} The set $X$ is strongly
		pointwise differentiable of order~$1$ [order~$\gamma$] at $x$,
		$\Tan (X, x ) = S$, and there exists a polynomial function $P
		: S \to S^\perp$ of degree at most $k$ satisfying
		\begin{gather*}
			\lim_{r \to 0+} r^{-k} \sup | f-P | \lIm X \cap \cball
			xr \rIm = 0 \quad \text{if $\alpha = 0$}, \\
			\limsup_{r \to 0+} r^{-k-\alpha} \sup | f-P | \lIm X
			\cap \cball xr \rIm < \infty \quad \text{if $\alpha >
			0$}.
		\end{gather*}
	\end{enumerate}
	In this case $\Der^i ( P \circ \project S ) (a) = \pt \Der^i A (a,S)$
	for $i = 0, \ldots, k$.
\end{corollary}

\begin{proof}
	As \eqref{item:eq_diff_sets:set2} implies the existence of $r > 0$, $0
	\leq \kappa < \infty$, and $s > 0$ such that
	\begin{equation*}
		A \cap \cball ar \subset \big \{ \chi \with | \perpproject{S}
		( \chi - a ) | \leq \kappa | \project S ( \chi - a ) | \big
		\}, \quad A \cap \project S^{-1} \lIm \cball {\project S (a)}
		s \rIm \subset \cball ar
	\end{equation*}
	by \ref{miniremark:cone} and the continuity of $f$ at $x$
	respectively, the equivalence and the postscript follow from
	\ref{remark:same_tangent_cone} and
	\ref{thm:eq_diff_sets}\,\eqref{item:eq_diff_sets:add_manifold}\,\eqref{item:eq_diff_sets:add_fct}.
\end{proof}

\begin{remark} \label{remark:cmp_fct_set}
	If $X$ is a neighbourhood of $x$ in $S$, then
	\eqref{item:eq_diff_sets:fct} is equivalent to the requirement that
	$f$ is pointwise differentiable of order~$\gamma$ at $x$.
\end{remark}

\begin{remark}
	In view of \ref{miniremark:cone} and
	\ref{thm:eq_diff_sets}\,\eqref{item:eq_diff_sets:add_fct}, it is clear
	that if $f$ is not continuous at~$x$ and $0 < \vdim < \adim$ it may
	happen that \eqref{item:eq_diff_sets:set2} is satisfied but
	$\limsup_{\chi \to x} |f(\chi)| = \infty$ in which case
	\eqref{item:eq_diff_sets:fct} does not hold.
\end{remark}

\begin{example} \label{example:alexandrov}
	If $A$ is an $\vdim+1$ dimensional convex subset of $\rel^\adim$, then
	the relative boundary of $A$ is strongly pointwise differentiable of
	order~$2$ at $\mathscr{H}^\vdim$~almost all of its points; in fact,
	one may employ Alexandrov's theorem, see for instance
	\cite[Theorem~6.9]{MR3409135}, to deduce this from
	\ref{corollary:cmp_fct_set} and \ref{remark:cmp_fct_set}.
\end{example}

\begin{lemma} \label{lemma:converge_to_hom_poly_fct}
	Suppose $k$ and $\adim$ are positive integers, $\vdim$ is an integer,
	$0 \leq \vdim \leq \adim$, $S \in \grass \adim \vdim$, $A \subset
	\rel^\adim$, $P : S \to S^\perp$ is a homogeneous polynomial
	function of degree~$k$, $B = \{ \chi + P ( \chi ) \with \chi \in S
	\}$, and $\beta_s : \rel^\adim \to \rel^\adim$ satisfy
	\begin{equation*}
		\beta_s (x) = s^{-1} \project S (x) + s^{-k} \perpproject S
		(x) \quad \text{for $x \in \rel^\adim$ and $0 < s < \infty$}.
	\end{equation*}

	Then the following two conditions are equivalent.
	\begin{enumerate}
		\item \label{item:convergence_to_hom_poly_fct:diff} The set
		$A$ is pointwise differentiable of order~$k$ at $0$ and
		\begin{equation*}
			\pt \Der^i A(0,S) = \Der^i ( P \circ \project S ) (0)
			\quad \text{for $i = 0, \ldots, k$}.
		\end{equation*}
		\item \label{item:convergence_to_hom_poly_fct:blow_up} There
		holds
		\begin{equation*}
			\lim_{s \to 0+} \dist (x,\beta_s \lIm A \rIm) = \dist
			(x,B) \quad \text{for $x \in \rel^\adim$}.
		\end{equation*}
	\end{enumerate}
\end{lemma}

\begin{proof}
	Notice that $\beta_s^{-1} = \beta_{1/s}$ and $\beta_s \lIm B \rIm = B$
	for $0 < s < \infty$ and
	\begin{equation*}
		\Lip \beta_s \leq s^{-k}, \quad \beta_{1/s} \lIm \cball 0r
		\rIm \subset \cball 0{sr}, \quad \cball 0{r/s} \subset \beta_s
		\lIm \cball 0r \rIm,
	\end{equation*}
	whenever $0 < s \leq 1$ and $0 < r < \infty$, in particular
	\begin{gather*}
		\dist ( \beta_s ( a), B ) \leq s^{-k} \dist (a,B) \quad
		\text{for $a \in A$}, \\
		\sup \dist ( \cdot, B ) \big \lIm \beta_s \lIm A \rIm \cap
		\cball 0r \big \rIm \leq s^{-k} \sup \dist ( \cdot, B ) \lIm A
		\cap \cball 0 {sr} \rIm.
	\end{gather*}

	Suppose \eqref{item:convergence_to_hom_poly_fct:diff} holds.  To prove
	\eqref{item:convergence_to_hom_poly_fct:blow_up}, one may assume
	\begin{equation*}
		A \subset \big \{ \chi \with | \perpproject S ( \chi ) | \leq
		\kappa | \project S (\chi) | \big \} \quad \text{for some $0
		\leq \kappa < \infty$}
	\end{equation*}
	by \ref{remark:same_tangent_cone} and \ref{miniremark:cone}.  Whenever
	$0 < r < \infty$ one estimates
	\begin{multline*}
		\sup \dist ( \cdot, \beta_s \lIm A \rIm ) \lIm B \cap \cball
		0r \rIm \leq \lambda s^{-1} \sup \dist ( \cdot, \project S
		\lIm A \rIm) \lIm S \cap \cball 0 {sr} \rIm \\
		+ s^{-k} \sup \big \{ | \perpproject S(a) - P ( \project S (a)
		) | \with a \in A, | \project S (a) | \leq 2sr \big \}
	\end{multline*}
	for $0 < s \leq 1$, where $\lambda = 1 + \Lip ( P | \cball 0{2r} )$;
	in fact, it is sufficient to notice
	\begin{gather*}
		b-\beta_s(a) = \project{S} ( b-s^{-1} a ) + P ( \project{S}
		(b)) - P ( \project S (s^{-1} a ) ) + s^{-k} ( P ( \project S
		(a)) - \perpproject S ( a )), \\
		| b- \beta_s(a)| \leq \lambda s^{-1} | \project S(sb) -
		\project{S} ( a ) | + s^{-k} | \perpproject S(a) - P (
		\project S(a) ) |
	\end{gather*}
	whenever $b \in B \cap \cball 0r$, $a \in A$, and $|\project S(a)|
	\leq 2sr$, as $0 \in \Clos A$ by \ref{remark:same_tangent_cone} implies
	\begin{equation*}
		0 \in \Clos \project S \lIm A \rIm, \quad \dist ( \project S
		(sb), \project S \lIm A \rIm \cap \cball 0 {2sr}) = \dist (
		\project S (sb), \project S \lIm A \rIm ).
	\end{equation*}
	In conjunction with the first paragraph,
	\eqref{item:convergence_to_hom_poly_fct:blow_up} now follows using
	\ref{miniremark:diff_sets} and
	\ref{thm:eq_diff_sets}\,\eqref{item:eq_diff_sets:add_fct}.

	Conversely, suppose \eqref{item:convergence_to_hom_poly_fct:blow_up}
	holds.  In order to prove
	\eqref{item:convergence_to_hom_poly_fct:diff}, firstly notice that $0
	\in \Clos A$ as $0 \in B$.  Next, the following assertion will be
	shown:
	\begin{gather*}
		A \cap \cball 0r \subset \big \{ \chi \with | \perpproject S (
		\chi ) | \leq \kappa | \project S ( \chi ) | \big \} \quad
		\text{for some $r > 0$ and $0 \leq \kappa < \infty$}, \\
		\lim_{A \owns x \to 0} | \perpproject S (x) - P ( \project S
		(x)) | / | \project S (x) |^k = 0.
	\end{gather*}
	For this purpose let $\varepsilon > 0$ and $C = \rel^\adim \cap \big
	\{ x \with | \perpproject{S} (x) - P ( \project S (x) ) | \geq
	\varepsilon | \project S (x) |^k \big \}$.  Clearly, $B \cap C = \{ 0
	\}$ and $\beta_s \lIm C \rIm = C$ for $0 < s < \infty$.  Since, by
	\cite[\hyperlink{2_10_21}{2.10.21}]{MR41:1976},
	\begin{equation*}
		\dist ( \cdot, \beta_s \lIm A \rIm ) \to
		\dist ( \cdot, B ) \quad \text{uniformly on compact subsets of
		$\rel^\adim$ as $s \to 0+$}
	\end{equation*}
	there exists $0 < s < \infty$ with $\beta_t \lIm A \rIm \cap C
	\cap \{ x \with |x| = 1 \} = \varnothing$ for $0 < t \leq s$, hence
	\begin{equation*}
		A \cap \beta_{1/s} \lIm \cball 01 \rIm \cap C \subset \{ 0 \}
	\end{equation*}
	as $\beta_{1/s} \lIm \cball 01 \rIm \without \{ 0 \} = \bigcup \big \{
	\beta_{1/t} \lIm \rel^\adim \cap \{ x \with |x| = 1 \} \rIm \with 0 <
	t \leq s \big \}$, and the assertion follows.  In particular, one may
	assume $A \subset \big \{ \chi \with | \perpproject S ( \chi ) | \leq
	\kappa | \project S ( \chi ) | \big \}$.  Noting
	\begin{gather*}
		s^{-1} \dist ( \project S(sx), \project S \lIm A \rIm ) \leq |
		\project S (x-\beta_s(a)) | \leq | x- \beta_s (a) | \quad
		\text{for $a \in A$, $x \in \rel^\adim$}, \\
		s^{-1} \sup \dist ( \cdot, \project S \lIm A
		\rIm ) \lIm S\cap \cball 0s \rIm \leq \sup \dist ( \cdot,
		\beta_s \lIm A \rIm ) \big \lIm B \cap \project
		S^{-1} \lIm \cball 01 \rIm \big \rIm
	\end{gather*}
	for $0 < s < \infty$, the set $\project S \lIm A \rIm$ is strongly
	pointwise differentiable of order~$1$ at $0$ with $\Tan ( \project S
	\lIm A \rIm, 0) = S$ by \ref{miniremark:diff_sets} and
	\ref{remark:same_tangent_cone}.  Consequently,
	\ref{thm:eq_diff_sets}\,\eqref{item:eq_diff_sets:add_fct} yields
	\eqref{item:convergence_to_hom_poly_fct:diff}. 
\end{proof}

\begin{remark} \label{remark:char_pt_diff_order_1}
	It follows by \ref{remark:same_tangent_cone}, \ref{miniremark:cone},
	and \ref{thm:eq_diff_sets}\,\eqref{item:eq_diff_sets:add_manifold}
	that a subset $A$ of $\rel^\adim$ is pointwise differentiable of
	order~$1$ at $a$ if and only if $A_s = \rel^\adim \cap \{ x \with a +
	sx \in A \}$ satisfy
	\begin{equation*}
		\lim_{s \to 0+} \dist (x,A_s) = \dist (x,T) \quad
		\text{whenever $x \in \rel^\adim$}
	\end{equation*}
	for some subspace $T$ of $\rel^\adim$.  In this case $T = \Tan(A,a)$.
	$\big ($But this condition on $\Tan(A,a)$ does not imply pointwise
	differentiability of order~$1$ of $A$ at $a$.$\big)$
\end{remark}

\begin{miniremark} \label{miniremark:graphical_planes}
	Suppose $\adim$ is a positive integer, $\vdim$ is an integer, $0 \leq
	\vdim \leq \adim$, $S \in \grass \adim \vdim$, $T \subset \rel^\adim$,
	$h \in \Hom ( S, S^\perp )$, and $L = \id{\rel^\adim} - h \circ
	\project S$.  Then $S^\perp \cap T = \{ 0 \}$ if and only if $S^\perp
	\cap L \lIm T \rIm = \{ 0 \}$.
\end{miniremark}

\begin{lemma} \label{lemma:substraction_lemma}
	Suppose $k$ and $\adim$ are positive integers, $\vdim$ is an integer,
	$0 \leq \vdim \leq \adim$, $A \subset \rel^\adim$, $S \in \grass \adim
	\vdim$, $f : S \to S^\perp$ of class~$k$, $B = \{ x - f(\project S(x))
	\with x \in A \}$, $a \in \rel^\adim$, and $b = a-f(\project{S}(a))$.

	Then the following two statements are equivalent.
	\begin{enumerate}
		\item The set $A$ is pointwise differentiable of order~$k$ at
		$a$, $\Tan (A,a) \in \grass \adim \vdim$, and $S^\perp \cap
		\Tan (A,a) = \{ 0 \}$.
		\item The set $B$ is pointwise differentiable of order~$k$ at
		$b$, $\Tan(B,b) \in \grass \adim \vdim$, and $S^\perp \cap
		\Tan (B,a) = \{ 0 \}$.
	\end{enumerate}
	In this case $\pt \Der^i A ( a,S ) = \pt \Der^i B(b,S) + \Der^i ( f
	\circ \project S) (a)$ for $i = 0, \ldots, k$.
\end{lemma}

\begin{proof}
	Define $g : \rel^\adim \to \rel^\adim$ by $g (x) = x - f ( \project S
	( x) )$ for $x \in \rel^\adim$.  Noting $\Tan (B,b) = \Der g (a) \lIm
	\Tan (A,a) \rIm$ by \cite[\hyperlink{3_1_21}{3.1.21}]{MR41:1976}, the
	principal conclusion follows from \ref{remark:diffeo_invariance} and
	\ref{miniremark:graphical_planes}.  The postscript then is a
	consequence of \ref{miniremark:cone} and
	\ref{thm:eq_diff_sets}\,\eqref{item:eq_diff_sets:add_fct}.
\end{proof}

\begin{theorem} \label{thm:inductive_blowup}
	Suppose $k$ and $\adim$ are positive integers, $\vdim$ is an integer
	with $0 \leq \vdim \leq \adim$, $a \in \mathbf R^n$, $A \subset
	\rel^\adim$, the set $A$ is pointwise differentiable of order~$1$ at
	$a$, $S \in \grass \adim \vdim$, and $P_1 \in \Hom ( S, S^\perp )$
	satisfies $\Tan (A,a) = \{ \chi + P_1 ( \chi ) \with \chi \in S \}$.

	Then the following two conditions are equivalent.
	\begin{enumerate}
		\item \label{item:inductive_blowup:diff} The set $A$ is
		pointwise differentiable of order~$k$ at $a$.
		\item \label{item:inductive_blowup:blowup} There exist
		homogeneous polynomial functions $P_i : S \to S^\perp$ of
		degree~$i$ for $i = 2, \ldots, k$ such that the following
		condition holds: If
		\begin{equation*}
			A_1 = \{ x - a \with x \in A \}, \quad A_i = \{ x -
			P_{i-1} ( \project S (x) ) \with x \in A_{i-1} \}
		\end{equation*}
		for $i =2, \ldots, k$, then $B_i = \{ \chi + P_i ( \chi )
		\with \chi \in S \}$ satisfy
		\begin{equation*}
			\lim_{s \to 0+} \dist (x, \beta_{i,s} \lIm A_i \rIm) =
			\dist (x,B_i) \quad \text{for $x \in \rel^\adim$ and
			$i = 2, \ldots, k$},
		\end{equation*}
		where $\beta_{i,s} : \rel^\adim \to \rel^\adim$ and
		$\beta_{i,s} (x) = s^{-1} \project S (x)+s^{-i}
		\perpproject{S} (x)$ for $x \in \rel^\adim$.
	\end{enumerate}
	In this case $\pt \Der^i A(a,S) = \Der^i (P_i \circ \project S) (0)$
	for $i = 1, \ldots, k$.
\end{theorem}

\begin{proof}
	Assume $a = 0$.  Notice that $\pt \Der A (0,S) = \Der ( P_1
	\circ \project S ) (0)$ by \ref{lemma:converge_to_hom_poly_fct} and
	\ref{remark:char_pt_diff_order_1}.

	Suppose \eqref{item:inductive_blowup:diff} holds and define $P_i : S
	\to S^\perp$ by
	\begin{equation*}
		P_i(\chi) = \langle \chi^i/i!, \pt \Der^i A(0,S) \rangle \quad
		\text{for $\chi \in S$  and $i = 2, \ldots, k$}.
	\end{equation*}
	Then \ref{lemma:substraction_lemma} inductively implies that the sets
	$A_i$ are pointwise differentiable of order~$k$ at $0$ and
	\begin{equation*}
		\pt \Der^j A_i (0,S) = \tsum{l=i}k \Der^j ( P_l \circ \project
		S ) ( 0 ) \quad \text{for $i = 1, \ldots, k$ and $j = 0,
		\ldots, k$},
	\end{equation*}
	hence applying \ref{lemma:converge_to_hom_poly_fct} with $k$, $A$, and
	$P$ replaced by $i$, $A_i$, and $P_i$ yields
	\eqref{item:inductive_blowup:blowup}.

	Now, suppose \eqref{item:inductive_blowup:blowup} holds.  Then $A_i$
	is pointwise differentiable of order~$i$ and
	\begin{equation*}
		\pt \Der^j A_i (0,S) = \Der^j ( P_i \circ \project S ) (0)
		\quad \text{for $j = 0, \ldots, i$ and $i = 1, \ldots, k$}
	\end{equation*}
	by \ref{lemma:converge_to_hom_poly_fct}, hence
	\ref{lemma:substraction_lemma} yields
	\eqref{item:inductive_blowup:diff} and the postscript.
\end{proof}

\begin{remark} \label{remark:equiv_convergence_sets}
	The convergence in \eqref{item:inductive_blowup:blowup} is equivalent
	(by \ref{miniremark:diff_sets} and
	\cite[\hyperlink{2_10_21}{2.10.21}]{MR41:1976}) to ``convergence
	locally in Hausdorff distance'' of $\Clos \beta_{i,s} \lIm A_i \rIm$
	to $B_i$ as $s \to 0+$, see David \cite[3.1]{MR1979936} for a
	definition in the case of sequences of closed sets, and to
	``Kuratowski convergence'' of the same sets, see Beer
	\cite[Theorem~1]{MR801600}.
\end{remark}

\section{Higher order differentiability theory for functions}
\label{sect:diff_fcts}

This section provides the main differentiability theorem for functions in
\ref{thm:pt_diff_functions} which serves as a model for the case of sets
treated in Section~\ref{sect:diff_sets}.  In fact, the theorems employed in
its proof, \ref{thm:borel_maps}, \ref{thm:isakov}, and
\ref{thm:big_O_little_o}, will also be used to treat the case of sets.

\begin{theorem} [classical] \label{thm:borel_maps}
	Suppose $X$ and $Y$ are complete, separable metric spaces, and
	$f$ is a function mapping a subset of $X$ into $Y$.
	
	Then the following two statements are equivalent.
	\begin{enumerate}
		\item \label{item:borel_maps:borel_map} The domain of $f$ is a
		Borel subset of $X$ and $f$ is a Borel function.
		\item \label{item:borel_maps:borel_graph} The set $f$ is a
		Borel subset of $X \times Y$.
	\end{enumerate}
\end{theorem}

\begin{proof}
	Defining $p : X \times Y \to X$ and $q : X \times Y$ by
	$p(x,y)=x$ and $q(x,y)=y$ for $(x,y) \in X \times Y$, one notices that
	$f^{-1} \lIm B \rIm = p \big \lIm f \cap q^{-1} \lIm B \rIm \big \rIm$
	for $B \subset Y$, hence \eqref{item:borel_maps:borel_graph} implies
	\eqref{item:borel_maps:borel_map} by
	\cite[\hyperlink{2_2_10}{2.2.10}]{MR41:1976}
	as $p | f$ is univalent.  The converse is
	elementary.
\end{proof}

\begin{theorem} \label{thm:isakov}
	Suppose $k$ is a nonnegative integer, $\vdim$ and $\adim$ are positive
	integers, $\vdim < \adim$, $0 \leq \alpha \leq 1$, $\gamma = k$ if
	$\alpha = 0$ and $\gamma = (k,\alpha)$ if $\alpha > 0$, $D$ is
	$\mathscr{L}^\vdim$ measurable, $f_i : D \to \bigodot^i ( \rel^\vdim,
	\rel^\codim )$ are $\mathscr{L}^\vdim \restrict D$ measurable for $i =
	0, \ldots, k$, and $\mathscr{L}^\vdim$~almost all $a \in D$ satisfy
	\begin{gather*}
		\lim_{r \to 0+} r^{-k} \sup | f-P_a | \lIm D \cap \cball ar
		\rIm = 0 \quad \text{if $\alpha = 0$}, \\
		\limsup_{r \to 0+} r^{-k-\alpha} \sup | f-P_a | \lIm D \cap
		\cball ar \rIm < \infty \quad \text{if $\alpha > 0$},
	\end{gather*}
	where $P_a : \rel^\vdim \to \rel^\codim$ and $P_a (x) = \sum_{i=0}^k
	\langle (x-a)^i/i!, f_i(a) \rangle$ for $x \in \rel^\vdim$.

	Then there exists a sequence of functions $g_j : \rel^\vdim \to
	\rel^\codim$ of class~$\gamma$ such that $\mathscr{L}^\vdim \big ( D
	\without \bigcup_{j=1}^\infty \{ x \with f_0 (x) = g_j(x) \} \big ) =
	0$.
\end{theorem}

\begin{proof}
	Assuming $D$ to be a Borel subset of $\rel^\vdim$ with
	$\mathscr{L}^\vdim ( D ) < \infty $ and noting
	\cite[\hyperlink{2_10_19}{2.10.19}\,(4)]{MR41:1976}, this follows from
	Isakov \cite[Theorem~1]{MR897693-english} if $\alpha = 0$ and Isakov
	\cite[Theorem~2]{MR897693-english}%
	\begin{footnote}
		{As the proof of \cite[Theorem~2]{MR897693-english} is omitted
		in that reference as ``completely analogous'' to
		\cite[Theorem~1]{MR897693-english}, the reader may find it
		helpful to notice that the presently needed case of
		\cite[Theorem~2]{MR897693-english} is in fact simpler than
		\cite[Theorem~1]{MR897693-english} provided one refers to
		\cite[\printRoman 6.2.2.2, \printRoman 6.2.3.1--3]{MR0290095}
		instead of \cite[\hyperlink{3_1_14}{3.1.14}]{MR41:1976} for
		the Whitney type extension theorem.}
	\end{footnote}
	and \ref{miniremark:class_ka} if $\alpha > 0$.
\end{proof}

\begin{remark}
	Isakov in fact provides a characterisation of the property in the
	conclusion (called ``$\mathscr{C}^{k+\alpha}$-rectifiability'' by
	Anzellotti and Serapioni in \cite[2.4]{MR1285779}) in terms of
	``approximate differentiability of the jet $f_i$, $i = 0, \ldots,
	k$'', see \cite{MR897693-english}.  A further characterisation using
	approximate partial derivatives of higher order is given by Lin and
	Liu \cite[1.5]{MR3164847}.
\end{remark}

\begin{theorem} \label{thm:big_O_little_o}
	Suppose $g : \rel^\vdim \to \{ t \with 0 \leq t \leq \infty \}$ and $0
	< l < \infty$.

	Then $\mathscr{L}^\vdim$ almost all $a$ satisfy that
	\begin{equation*}
		\limsup_{r \to 0+} r^{-l} \sup g \lIm \oball ar \rIm \quad
		\text{equals either $0$ or $\infty$}.
	\end{equation*}
\end{theorem}

\begin{proof}
	Define closed sets $E_i = \rel^\vdim \cap \{ y \with \text{$\sup g
	\lIm \oball yr \rIm \leq i r^l$ for $0 < r < 1/i$} \}$ for every
	positive integer $i$.  It will be sufficient to show
	\begin{equation*}
		\lim_{r \to 0+} r^{-l} \sup g \lIm \oball ar \rIm = 0 \quad
		\text{for $\mathscr{L}^\vdim$ almost all $a \in E_i$}.
	\end{equation*}

	Suppose $a \in E_i$ and $0 < r < 1/i$.

	Whenever $x \in \oball ar$ one chooses $y \in E_i$ with $|x-y| = \dist
	(x,E_i)$ and infers that $|x-y| \leq |x-a| < 1/i$ and $g(x) \leq i
	|y-x|^l$.  This implies
	\begin{equation*}
		r^{-l} \sup g \lIm \oball ar \rIm \leq i \big ( r^{-1} \sup
		\dist ( \cdot, E_i ) \lIm \oball ar \rIm \big )^l.
	\end{equation*}
	As $r$ approaches $0$ the right hand side of the preceding inequality
	approaches~$0$ provided $E_i$ is pointwise differentiable of order~$1$
	at~$a$ and $\Tan (E_i,a) = \rel^\vdim$ which is true at
	$\mathscr{L}^\vdim$ almost all $a \in E_i$ by
	\ref{example:subset_of_plane}.
\end{proof}

\begin{remark}
	Taking $\adim = \vdim$, $U = \rel^\vdim$, $V$ the varifold associated
	to $\mathscr{L}^\vdim$, and $q = l$, this is a special case of
	Kolasiński and the author \cite[4.4]{MR3625810}.  The proof
	is included for the convenience of the reader and is modelled on
	\cite[\hyperlink{2_9_17}{2.9.17}]{MR41:1976}.  A similar result with
	certain Lebesgue seminorms of $g$ replacing ``$\sup g \lIm \oball ar
	\rIm$'' was obtained by Calder{\'o}n and Zygmund in
	\cite[Theorem~10\,(ii)]{MR0136849}.  Further variants and comments on
	the history and the sharpness of general results of this type are
	contained in \cite[3.1--3.4]{snulmenn.isoperimetric} and
	\cite[4.1--4.5]{MR3625810}.
\end{remark}

\begin{theorem} \label{thm:pt_diff_functions}
	Suppose $k$ is a nonnegative integer, $\vdim$ and $\adim$ are positive
	integers, $\vdim < \adim$, $0 \leq \alpha \leq 1$, $\gamma = k$ if
	$\alpha = 0$ and $\gamma = (k,\alpha)$ if $\alpha > 0$, $U$ is an open
	subset of $\rel^\vdim$, $f : U \to \rel^\codim$, and $X$ is the set of
	$x$ at which $f$ is pointwise differentiable of order~$\gamma$.

	Then the following four statements hold.
	\begin{enumerate}
		\item \label{item:pt_diff_functions:borel} The functions $\pt
		\Der^i f$ are Borel functions whose domains are Borel subsets
		of $\rel^\vdim$ for $i = 0, \ldots, k$ and $X$ is a Borel
		subset of $\rel^\vdim$.
		\item \label{item:pt_diff_functions:rect} There exists a
		sequence of functions $g_j : \rel^\vdim \to \rel^\codim$ of
		class~$\gamma$ such that $\mathscr{L}^\vdim \big ( X \without
		\bigcup_{j=1}^\infty \{ x \with f(x) = g_j(x) \} \big ) = 0$.
		\item \label{item:pt_diff_functions:little_o} If $g :
		\rel^\vdim \to \rel^\codim$ is of class~$\gamma$ and $Y = \{ y
		\with f(y)=g(y) \}$, then
		\begin{gather*}
			\pt \Der^i f (a) = \Der^i g(a) \quad \text{for $i = 0,
			\ldots, k$}, \\
			\lim_{x \to a} |f(x)-g(x)|/|x-a|^{k+\alpha} = 0
		\end{gather*}
		at $\mathscr{L}^\vdim$ almost all $a \in X \cap Y$.
		\item \label{item:pt_diff_functions:rademacher} If $\alpha =
		1$, then $f$ is pointwise differentiable of order~$k+1$ at
		$\mathscr{L}^\vdim$ almost all $a \in X$.
	\end{enumerate}
\end{theorem}

\begin{proof}
	Assume $U = \rel^\vdim$.

	Abbreviate $C = \rel^\vdim \times \prod_{l=0}^k \bigodot^l (
	\rel^\vdim, \rel^\codim )$.  Whenever $i$ and $j$ are positive
	integers define $D_{i,j}$ to consist of those
	$(a,\phi_0,\ldots,\phi_k) \in C$ satisfying
	\begin{equation*}
		\big | f(x)- \tsum{l=0}k \langle (x-a)^l/l!, \phi_l \rangle
		\big | \leq |x-a|^k/i \quad \text{for $x \in \oball a{1/j}$}
	\end{equation*}
	and $E_i$ to consist of those $(a,\phi_0,\ldots,\phi_k) \in C$
	satisfying
	\begin{equation*}
		\big | f(x)- \tsum{l=0}k \langle (x-a)^l/l!, \phi_l \rangle
		\big | \leq i |x-a|^{k+\alpha} \quad \text{for $x \in \oball
		a{1/i}$}.
	\end{equation*}
	Furthermore, define
	\begin{equation*}
		F = \big \{ (a,\pt \Der^0 f (a), \ldots, \pt \Der^k f(a))
		\with \text{$a \in \dmn \pt \Der^kf$ and $a \in X$} \big \}.
	\end{equation*}
	Since the sets $D_{i,j}$ and $E_i$ are closed and
	\begin{equation*}
		F = \bigcap_{i=1}^\infty \bigcup_{j=1}^\infty C_{i,j} \quad
		\text{if $\alpha = 0$}, \qquad F = \bigcup_{i=1}^\infty E_i
		\quad \text{if $\alpha > 0$},
	\end{equation*}
	$F$ is a Borel set.  Noting that the condition $a \in X$ in the
	definition of $F$ is redundant if $\alpha =0$, one infers that $\pt
	\Der^k f$ is a Borel function whose domain is a Borel set from
	\ref{thm:borel_maps}, hence the same holds for $\pt \Der^i f$ for $i =
	0, \ldots, k-1$.  As
	\begin{equation*}
		X = \big \{ a \with ( a , \pt \Der^0 f(a), \ldots, \pt \Der^k
		f (a)) \in F \big \},
	\end{equation*}
	$X$ is a Borel set and \eqref{item:pt_diff_functions:borel} follows.

	\eqref{item:pt_diff_functions:borel} and \ref{thm:isakov} yield
	\eqref{item:pt_diff_functions:rect}.  To prove the first half of
	\eqref{item:pt_diff_functions:little_o}, it is sufficient to apply
	\ref{thm:poly_uniqueness} with $S$, $Y$, $\Der^i P(a)$, and $X$
	replaced by $\rel^\vdim$, $\rel^\codim$, $\pt \Der (f-g)(a)$, and $Y$
	for $i = 0, \ldots, k$ whenever $a \in X$, the set $Y$ is pointwise
	differentiable of order~$1$ at~$a$, and $\Tan (Y,a) = \rel^\vdim$ by
	\ref{example:subset_of_plane}.  To deduce the second half of
	\eqref{item:pt_diff_functions:little_o} from its first half, one may
	assume $\alpha > 0$ and hence apply \ref{thm:big_O_little_o} with $g$
	and $l$ replaced by $|f-g|$ and $k+\alpha$.  To prove
	\eqref{item:pt_diff_functions:rademacher}, notice that in this case
	the functions~$g_j$ in \eqref{item:pt_diff_functions:little_o} may be
	required to be of class~$k+1$ by
	\cite[\hyperlink{3_1_15}{3.1.15}]{MR41:1976}, hence
	\eqref{item:pt_diff_functions:little_o} implies
	\eqref{item:pt_diff_functions:rademacher}.
\end{proof}

\begin{remark} \label{remark:set_of_continuity_points}
	In view of \ref{remark:pt_diff_fct_low_order}, the case $k = 0$ of
	\eqref{item:pt_diff_functions:borel} merely restates that the set of
	continuity points of $f$ is a Borel set, see for instance
	\cite[(6.90)\,(a)--(c)]{MR0367121}.
\end{remark}

\begin{remark} \label{remark:pt_diff_functions}
	The proof of \eqref{item:pt_diff_functions:borel} is a variation of
	the argument of \cite[\hyperlink{3_1_1}{3.1.1}]{MR41:1976} where the
	case $k=1$ is treated for continuous $f$.  The case $k=1$ for
	arbitrary $f$ occurs in Járai \cite{MR812302}.  In fact, it is
	shown there that $(\Der f)^{-1} \lIm C \rIm$ belongs to the class
	``$\mathscr{F}_{\sigma\delta}$'' whenever $C$ is a closed subset of
	$\Hom ( \rel^\vdim, \rel^\codim )$.  That class is also denoted
	``$\boldsymbol{\Pi}_3^0 ( \rel^\vdim )$'' within the Borel hierarchy,
	see \cite{MR1321597}.
\end{remark}

\begin{remark}
	The pattern proof of
	\eqref{item:pt_diff_functions:rect}--\eqref{item:pt_diff_functions:rademacher}
	is that of Liu \cite[1.6]{MR2414427}.
\end{remark}

\begin{remark}
	A related study of pointwise differentials of order $k$ in terms of
	general difference quotients is carried out by Zibman, see
	\cite{MR508884}.
\end{remark}

\section{Higher order differentiability theory for sets}
\label{sect:diff_sets}

In this section Theorem~\hyperlink BB is proven, see \ref{thm:diff_order_one}
and \ref{thm:derivative_Borel} for the Borel properties asserted in~\hyperlink
BB\,\eqref{intro:rademacher_sets:borel} and \ref{thm:rademacher_sets} for the
differentiability results asserted in~\hyperlink
BB\,\eqref{intro:rademacher_sets:approx}--\eqref{intro:rademacher_sets:points}.

\begin{miniremark} \label{miniremark:planes_over_planes}
	Suppose $\adim$ is a positive integer, $\vdim$ is an integer,
	and $0 \leq \vdim \leq \adim$.  Then the set
	\begin{equation*}
		( \grass \adim \vdim  \times \grass \adim \vdim ) \cap \big \{
		(S,T) \with S^\perp \cap T = \{ 0 \} \big \}
	\end{equation*}
	is an open subset of $\grass \adim \vdim \times \grass \adim \vdim$;
	in fact, $S^\perp \cap T = \{ 0 \}$ if and only if $\bigwedge_\vdim (
	\project T \circ \project S ) \neq 0$ for $S, T \in \grass \adim
	\vdim$.
\end{miniremark}

\begin{theorem} \label{thm:diff_order_one}
	Suppose $A \subset \rel^\adim$ and
	\begin{equation*}
		\tau =  \{ (a,S) \with S = \Tan (A,a) \} \cap \dmn \pt \Der A
		\subset \rel^\adim \times {\textstyle \bigcup_{\vdim=0}^\adim
		\grass \adim \vdim }.
	\end{equation*}

	Then $\tau$ is a Borel function whose domain is a Borel set and
	$\tau^{-1} \lIm \grass \adim \vdim \rIm$ is countably $\vdim$
	rectifiable whenever $\vdim = 0, \ldots, \adim$.
\end{theorem}

\begin{proof}
	Let $V = \rel^\adim$ and $G = \bigcup_{\vdim = 0}^\adim \grass \adim
	\vdim$.  Endowing $\rel^V$ with the Cartesian product topology,%
	\begin{footnote}
		{Whenever $X$ and $Y$ are sets $Y^X$ denotes set of maps from
		$X$ into $Y$, see \cite[p.~669]{MR41:1976}.}
	\end{footnote}
	one notices that $\delta : V \times G \to \rel^V$ defined by
	\begin{equation*}
		\delta ( a,S ) (x) = \dist (x,\{a+\chi \with \chi \in S \} )
		\quad \text{for $(a,S) \in V \times G$ and $x \in V$}
	\end{equation*}
	is continuous and defines closed sets $C_{i,j}$ consisting of those
	$(a,S) \in V \times G$ with
	\begin{equation*}
		a \in \Clos A, \qquad \sup | \dist (\cdot,A) - \delta (a,S) |
		\lIm \oball ar \rIm \leq r/i \quad \text{for $0 < r < 1/j$}
	\end{equation*}
	whenever $i$ and $j$ are positive integers.  Observe that
	\begin{equation*}
		\tau = \bigcap_{i=1}^\infty \bigcup_{j=1}^\infty C_{i,j},
	\end{equation*}
	by \ref{remark:same_tangent_cone}, \ref{miniremark:cone}, and
	\ref{thm:eq_diff_sets}\,\eqref{item:eq_diff_sets:add_manifold}, hence
	$\tau$ is a Borel set.  Since $\tau$ is a function, it follows that
	$\tau$ is a Borel function whose domain is a Borel set by
	\ref{thm:borel_maps}.

	Suppose $\vdim$ is an integer and $0 \leq \vdim \leq \adim$.  Choose a
	countable dense subset $D$ of $\grass \adim \vdim$, define $E(i,S)$ to
	consist of $a \in \Clos A$ such that
	\begin{equation*}
		A \cap \oball a{1/i} \subset \big \{ \chi \with |
		\perpproject{S} (\chi-a) \leq i | \project S (\chi-a) | \big
		\}
	\end{equation*}
	for $i = 1, 2, 3, \ldots$ and $S \in D$.  Noting
	\ref{remark:same_tangent_cone} and \ref{miniremark:cone}, one infers
	from \ref{miniremark:planes_over_planes} that
	\begin{equation*}
		{\textstyle \tau^{-1} \lIm \grass \adim \vdim \rIm \subset
		\bigcup \{ E(i,S) \with \text{$i = 1, 2, 3, \ldots$ and $S \in
		D$} \}}.
	\end{equation*}
	Therefore the set $\tau^{-1} \lIm \grass \adim \vdim \rIm$ is
	countably $\vdim$ rectifiable by
	\cite[\hyperlink{3_3_5}{3.3.5}]{MR41:1976}.
\end{proof}

\begin{lemma} \label{lemma:convergence}
	Suppose $B \subset \rel^\adim$ and $h_j$ is a sequence of univalent
	maps of $\rel^\adim$ onto $\rel^\adim$ satisfying $\lim_{j \to \infty}
	h_j(x)=x$ for $x \in \rel^\adim$, and
	\begin{equation*}
		\lim_{j \to \infty} \sup \big \{ | h_j^{-1} (\chi)-\chi |
		\with \chi \in K \big \} = 0 \quad \text{whenever $K$ is a
		compact subset of $\rel^\adim$}.
	\end{equation*}

	Then there holds
	\begin{equation*}
		\lim_{j \to \infty} \dist (x,h_j \lIm B \rIm) = \dist (x, B )
		\quad \text{for $x \in \rel^\adim$}.
	\end{equation*}
\end{lemma}

\begin{proof}
	Suppose $x \in \rel^\adim$.  Clearly, $\limsup_{j \to \infty} \dist
	(x, h_j \lIm B \rIm ) \leq \dist (x, B )$.  Moreover, if $\varepsilon
	> 0$ and $r = \dist (x,B) \geq 0$, $j$ is a positive integer, and
	\begin{equation*}
		|h_j^{-1} ( \chi ) - \chi | \leq \varepsilon \quad \text{for
		$\chi \in \cball xr$},
	\end{equation*}
	then $|h_j(b)-x| \geq r - \varepsilon$ for $b \in B$, as $|h_j(b)-x|
	\leq r$ implies $|b-h_j(b)| \leq \varepsilon$.  It follows that $\dist
	(x,B) \leq \liminf_{j \to \infty} \dist (x,h_j \lIm B \rIm)$.
\end{proof}

\begin{lemma} \label{lemma:varying_poly_fcts}
	Suppose $k$ and $\adim$ are positive integers, $G =
	\bigcup_{\vdim=0}^\adim \grass \adim \vdim$, $V = \rel^\adim$, $C$ is
	the set of
	\begin{equation*}
		{\textstyle (a,S,\phi_0, \ldots, \phi_k) \in V \times G \times
		\prod_{i=0}^k \bigodot^i ( V, V )}
	\end{equation*}
	such that $\phi_0 = \perpproject{S} (a)$ and
	\begin{equation*}
		\phi_i ( v_1, \ldots, v_i ) = \phi_i ( \project{S} (v_1),
		\ldots, \project{S} (v_i) ) \in S^\perp \quad \text{for $v_1,
		\ldots, v_i \in V$, $i = 1, \ldots, k$},
	\end{equation*}
	$\rel^V$ is endowed with the Cartesian product topology,
	and $\delta : C \to \rel^V$ satisfies
	\begin{align*}
		& \delta (a,S,\phi_0, \ldots, \phi_k) (x) = \dist (x, \{
		\chi+P(\chi) \with \chi \in S \} ) \\
		& \qquad \text{where $P (\chi) = \tsum{i=0}k \langle
		(\chi-\project S (a) )^i/i!, \phi_i \rangle$},
	\end{align*}
	whenever $(a,S,\phi_0, \ldots, \phi_k) \in C$ and $x \in V$.

	Then $C$ is a closed subset of $V \times G \times \prod_{i=0}^k
	\bigodot^i ( V, V )$ and $\delta$ is continuous.
\end{lemma}

\begin{proof}
	The first statement is trivial.

	To prove the second statement, suppose $(a_j,S_j,\phi_{0,j}, \ldots,
	\phi_{k,j} )$ is a sequence in $C$ converging to $(a,S,\phi_0, \ldots,
	\phi_k ) \in C$ as $j \to \infty$, $P_j$ and $P$ are the associated
	polynomial functions, and
	\begin{equation*}
		B_j = \{ \chi + P_j (\chi) \with \chi \in S_j \}, \quad B = \{
		\chi + P ( \chi ) \with \chi \in S \}.
	\end{equation*}
	One may assume that $S_j \in \grass \adim {\dim S}$ for every positive
	integer $j$.  Hence, as $\mathbf{O} (n)$ operates on the homogeneous
	space $\grass \adim {\dim S}$ by a transitive left action, see
	\cite[\hyperlink{2_7_1}{2.7.1},
	\hyperlink{3_2_28}{3.2.28}\,(2)\,(4)]{MR41:1976}, there exists a
	sequence $f_j \in \mathbf{O} (n)$ with $f_j \lIm S \rIm = S_j$ and
	$f_j \to \id{V}$ as $j \to \infty$.  Define $g_j : V \to V$ and $g : V
	\to V$ by
	\begin{equation*}
		g_j (x) = x - P_j ( \project{(S_j)} (x)), \quad g(x) =
		x-P(\project S(x))
	\end{equation*}
	whenever $x \in V$ and $j$ is a positive integer and notice
	that $g_j$ map $V$ univalently onto $V$ and
	\begin{equation*}
		g_j^{-1} (x) = x + P_j ( \project{(S_j)}(x)) \quad \text{for
		$x \in V$}.
	\end{equation*}
	Since $B_j = h_j \lIm B \rIm$ for $h_j = g_j^{-1} \circ f_j \circ
	g$, applying \ref{lemma:convergence} yields the conclusion.
\end{proof}

\begin{theorem} \label{thm:derivative_Borel}
	Suppose $k$ and $\adim$ are positive integers, $0 \leq \alpha \leq 1$,
	$\gamma = k$ if $\alpha = 0$ and $\gamma = (k,\alpha)$ if $\alpha
	> 0$, $A \subset \rel^\adim$, and $X$ is the set of $a \in \rel^\adim$
	such that $A$ is pointwise [strongly pointwise] differentiable of
	order~$\gamma$ at $a$.

	Then $\pt \Der^k A$ is a Borel function whose domain is a Borel set
	and $X$ is a Borel set.
\end{theorem}

\begin{proof}
	Let $C$ and $\delta$ as in \ref{lemma:varying_poly_fcts}.  Whenever
	$i$ and $j$ are positive integers define $C_{i,j}$ to consist of those
	$(a,S,\phi_0,\ldots,\phi_k) \in C$ with $a \in \Clos A$ satisfying
	\begin{equation*}
		\sup | \dist ( \cdot, A ) - \delta (a,S,\phi_0, \ldots,
		\phi_k) | \lIm \oball ar \rIm \leq r/i \quad \text{for $0 < r
		< 1/j$},
	\end{equation*}
	$D_{i,j}$ to consist of those $(a,S,\phi_0,\ldots,\phi_k) \in C$ with
	$a \in \Clos A$ satisfying
	\begin{gather*}
		\sup \delta (a,S,\phi_0, \ldots, \phi_k) \lIm A \cap \oball ar
		\rIm \leq r^k/i \quad \text{for $0 < r < 1/j$}, \\
		\left [ \sup | \dist ( \cdot, A ) - \delta (a,S,\phi_0,
		\ldots, \phi_k) | \lIm \oball ar \rIm \leq r^k/i \quad
		\text{for $0 < r < 1/j$}, \right ]
	\end{gather*}
	and $E_i$ to consist of those $(a,S,\phi_0,\ldots,\phi_k) \in C$ with
	$a \in \Clos A$ satisfying
	\begin{gather*}
		\sup \delta ( a,S, \phi_0, \ldots, \phi_k ) \lIm A \cap \oball
		ar \rIm \leq ir^k \quad \text{for $0 < r < 1/i$}. \\
		\left [ \sup | \dist (\cdot, A) -\delta ( a,S, \phi_0, \ldots,
		\phi_k ) | \lIm \oball ar \rIm \leq ir^k \quad
		\text{for $0 < r < 1/i$}. \right ]
	\end{gather*}
	Furthermore, define
	\begin{equation*}
		F = \big \{ \big (a,S, \pt \Der^0 A (a,S), \ldots, \pt \Der^k
		A(a,S) \big ) \with \text{$(a,S) \in \dmn \pt \Der^k A$, $a
		\in X$} \big \}
	\end{equation*}
	and notice that the condition $a \in X$ is redundant in the
	unbracketed case if $\alpha = 0$.  Since the sets $C_{i,j}$,
	$D_{i,j}$, and $E_i$ are closed by \ref{lemma:varying_poly_fcts} and
	\begin{gather*}
		F = \bigcap_{i=1}^\infty \bigcup_{j=1}^\infty ( C_{i,j} \cap
		D_{i,j}) \quad \text{if $\alpha = 0$}, \qquad
		F = \left ( \bigcap_{i=1}^\infty \bigcup_{j=1}^\infty C_{i,j}
		\right ) \cap \bigcup_{i=1}^\infty E_i  \quad \text{if $\alpha
		> 0$}
	\end{gather*}
	by \ref{remark:same_tangent_cone}, \ref{miniremark:cone}, and
	\ref{thm:eq_diff_sets}\,\eqref{item:eq_diff_sets:add_manifold}, $F$ is
	a Borel set. Consequently, \ref{thm:borel_maps} implies that $\pt
	\Der^i A$ are Borel functions whose domains are Borel sets for $i = 0,
	\ldots, k$.  As
	\begin{gather*}
		X = \big \{ a \with (a,\tau(a), \pt \Der^0 A ( a, \tau (a) ),
		\ldots, \pt \Der^k A ( a, \tau (a) ) ) \in F \big \},
	\end{gather*}
	one may use \ref{thm:diff_order_one} to deduce that $X$ is a Borel
	set.
\end{proof}

\begin{lemma} \label{lemma:subset_pt_diff_order_gamma}
	Suppose $k$ and $\adim$ are positive integers, $\vdim$ is an integer,
	$0 \leq \vdim \leq \adim$, $0 \leq \alpha \leq 1$, $\gamma = k$ if
	$\alpha = 0$, and $\gamma = (k,\alpha)$ if $\alpha > 0$, $A \subset
	\rel^\adim$, $a \in \Clos A$, the set $A$ is pointwise differentiable
	of order~$\gamma$ at $a$ with $\Tan (A,a) \in \grass \adim \vdim$, $D
	\subset S \in \grass \adim \vdim$, $f : S \to S^\perp$ is of
	class~$1$, $x = \project S (a)$, the set $D$ is pointwise
	differentiable of order~$1$ at $x$ with $\Tan (D,x) = S$, and
	\begin{equation*}
		B = \{ \chi + f (\chi) \with \chi \in D \} \subset \Clos A.
	\end{equation*}

	Then $B$ is pointwise differentiable of order~$\gamma$ at $a$ and
	\begin{gather*}
		\pt \Der^i A (a,\cdot) = \pt \Der^i B (a,\cdot) \quad
		\text{for $i = 0, \ldots, k$}.
	\end{gather*}
\end{lemma}

\begin{proof}
	Possibly replacing $A$ by $\Clos A$, one may assume $A$ to be closed,
	hence $a \in A$ and $B \subset A$.  Noting $x \in \Clos D$ by
	\ref{remark:same_tangent_cone}, one may also assume $x \in D$ possibly
	replacing $D$ by $D \cup \{ x \}$.  Define
	\begin{equation*}
		C = \big \{ \chi + \tsum{i=0}k \langle ( \chi-x )^i/i!, \pt
		\Der^i A (a,S) \rangle \with \chi \in S \big \}.
	\end{equation*}
	Applying \ref{corollary:cmp_fct_set} with $k$, $\alpha$, $X$, $f$, and
	$A$ replaced by $1$, $0$, $D$, $f | D$, and $B$, one obtains that $B$
	is pointwise differentiable of order~$1$ at~$a$, $\Tan ( B,a ) \in
	\grass \adim \vdim$, and $S^\perp \cap \Tan (B,a) = \{ 0 \}$, hence
	$\Tan (A,a) = \Tan (B,a)$.  In particular, one may assume that
	\begin{equation*}
		A \subset \big \{ \chi \with | \perpproject{S} ( \chi-a ) |
		\leq \kappa | \project S ( \chi-a ) | \big \} \quad \text{for
		some $0 \leq \kappa < \infty$}
	\end{equation*}
	by \ref{miniremark:cone}.  Therefore, noting
	\ref{remark:agreement_strong_normal}, twice applying
	\ref{thm:eq_diff_sets}\,\eqref{item:eq_diff_sets:add_half_manifold}
	with $(A,B)$ replaced by $(A,C)$ and $(B,C)$ respectively yields the
	conclusion.
\end{proof}

\begin{theorem} \label{thm:rademacher_sets}
	Suppose $k$ and $\adim$ are positive integers, $\vdim$ is an integer,
	$0 \leq \vdim \leq \adim$, $A \subset \rel^\adim$, $0 \leq \alpha \leq
	1$, $\gamma = k$ if $\alpha = 0$ and $\gamma = (k,\alpha)$ if $\alpha
	> 0$, and $X$ is the set of $a \in \rel^\adim$ such that $A$ is
	pointwise [strongly pointwise] differentiable of order $\gamma$ at $a$
	with $\dim \Tan (A,a) = \vdim$.

	Then the following three statements hold.
	\begin{enumerate}
		\item \label{item:rademacher_sets:approx} There exists a
		countable collection of $\vdim$~dimensional submanifolds
		of class~$\gamma$ of $\rel^\adim$ covering $\mathscr{H}^\vdim$
		almost all of $X$.
		\item \label{item:rademacher_sets:closeness} If $B$ is an
		$\vdim$ dimensional submanifold of class~$\gamma$ of
		$\rel^\adim$, then $\mathscr{H}^\vdim$ almost all $a \in B
		\cap X$ satisfy $\pt \Der^i A (a,\cdot) = \pt \Der^i B
		(a,\cdot)$ for $i = 0, \ldots, k$ and
		\begin{gather*}
			\lim_{r \to 0+} r^{-k-\alpha} \sup \dist ( \cdot, B )
			\lIm A \cap \cball ar \rIm =0. \\
			\left [ \lim_{r \to 0+} r^{-k-\alpha} \sup | \dist (
			\cdot, A ) - \dist ( \cdot, B ) | \lIm \cball ar \rIm
			= 0. \right ]
		\end{gather*}
		\item \label{item:rademacher_sets:points} If $\alpha = 1$,
		then $A$ is pointwise [strongly pointwise] differentiable of
		order~$k+1$ at $\mathscr{H}^\vdim$ almost all $a \in X$.
	\end{enumerate}
\end{theorem}

\begin{proof}
	Notice that $X$ is a countably $\vdim$ rectifiable Borel set by
	\ref{thm:diff_order_one} and \ref{thm:derivative_Borel} and $X \subset
	\Clos A$ by \ref{remark:same_tangent_cone}.  In view of
	\ref{remark:diff_sets_zero_dim} one may assume $\vdim > 0$.

	In order to prove \eqref{item:rademacher_sets:approx}, suppose $S \in
	\grass \adim \vdim$, $f : S \to S^\perp$ is of class~$1$, and define
	Borel sets $B = X \cap \{ \chi + f ( \chi ) \with \chi \in S \}$ and
	$D = S \cap \{ \chi \with \chi + f ( \chi ) \in X \}$.  If $a \in B$,
	$x = \project S(a)$, and $D$ is pointwise differentiable of order~$1$
	at~$x$ with $\Tan (D,x)=S$, then the set $B$ is pointwise
	differentiable of order~$\gamma$ at $a$ and
	\begin{equation*}
		\pt \Der^i A ( a, \cdot ) = \pt \Der^i B ( a, \cdot ) \quad
		\text{for $i = 0, \ldots, k$}
	\end{equation*}
	by
	\ref{lemma:subset_pt_diff_order_gamma}, hence
	\ref{corollary:cmp_fct_set} with $x$, $X$, $f$, and $A$ replaced by
	$\project S (a)$, $D$, $f|D$, and $B$ yields
	\begin{gather*}
		\lim_{r \to 0+} r^{-k} \sup | f-P_x | \lIm D \cap \cball x r
		\rIm = 0 \quad \text{if $\alpha = 0$}, \\
		\limsup_{r \to 0+} r^{-k-\alpha} \sup | f-P_x | \lIm D \cap
		\cball x r \rIm < \infty \quad \text{if $\alpha > 0$},
	\end{gather*}
	where $P_x : S \to S^\perp$ and $P_x ( \chi ) = \sum_{i=0}^k \langle (
	\chi - x )^i/i!, \pt \Der^i A (a,S) \rangle$ for $\chi \in S$.  Since
	these conditions hold for $\mathscr{H}^\vdim$ almost all $a \in B$
	with $x = \project S(a)$ by \ref{example:subset_of_plane}, there exist
	$g_j : S \to S^\perp$ of class~$\gamma$ such that
	\begin{equation*}
		{\textstyle \mathscr{H}^\vdim \big ( D \without
		\bigcup_{j=1}^\infty \{ x \with f(x)=g_j(x) \} \big ) = 0}
	\end{equation*}
	by \ref{thm:isakov} and \ref{thm:derivative_Borel}.  Consequently,
	\eqref{item:rademacher_sets:approx} follows from
	\cite[\hyperlink{3_2_29}{3.2.29}]{MR41:1976}.
	
	In order to prove \eqref{item:rademacher_sets:closeness}, it is
	sufficient to consider $B = \{ \chi + f ( \chi ) \with \chi \in S \}$
	corresponding to $S \in \grass \adim \vdim$ and $f : S \to S^\perp$ of
	class~$\gamma$ with $\Lip f < \infty$ by \ref{remark:hoelder_ok} and
	\cite[\hyperlink{3_1_19}{3.1.19}\,(5)]{MR41:1976}.  Define
	\begin{equation*}
		Y = ( B \cap X ) \cap \big \{ y \with \text{$\pt \Der^i A ( y,
		\cdot ) = \pt \Der^i B ( y, \cdot )$ for $i = 0, \ldots, k$}
		\big \}
	\end{equation*}
	and $D = S \cap \{ x \with x + f(x) \in X \}$.  In view of
	\ref{example:subset_of_plane}, twice applying
	\ref{lemma:subset_pt_diff_order_gamma} with $(A,B)$ replaced by $\big
	(A,B \cap \project{S}^{-1} \lIm D \rIm \big )$ and $\big (B,B \cap
	\project{S}^{-1} \lIm D \rIm \big )$ respectively yields
	\begin{equation*}
		\{ x + f (x) \with x \in S, \density^\vdim (
		\mathscr{H}^\vdim \restrict S \without \Clos D, x ) = 0 \}
		\subset Y, \quad \mathscr{H}^\vdim ( B \cap X \without Y ) =
		0.
	\end{equation*}
	Let $\kappa = \Lip f$ and $\lambda = 2^{-1} (1+\kappa)^{-1}$.  For
	each positive integer $j$ define
	\begin{equation*}
		C_j = \rel^\adim \cap \Big \{ a \with A \cap \oball a{1/j}
		\subset \big \{ \chi \with | \perpproject S ( \chi-a) | \leq j
		| \project{S} (\chi-a) | \big \} \Big \}
	\end{equation*}
	and notice that $Y \subset \bigcup_{j=1}^\infty C_j$ by
	\ref{miniremark:cone}. Next, define for $j = 1, 2, 3, \ldots$ sets
	\begin{equation*}
		A_j = A \cap \big \{ \chi \with | \perpproject S(\chi) - f (
		\project S ( \chi ) ) | < 1/(2j) \big \}, \quad D_j = \{ x
		\with x + f(x) \in C_j \cap Y \}
	\end{equation*}
	and functions $g_j : S \to \{ t \with 0 \leq t \leq \infty
	\}$ by
	\begin{multline*}
		g_j(x) = h(x), \qquad \Big [ g_j(x)= \sup \{ h(x), \dist
		(x+f(x),A_j) \} , \Big ] \\
		\text{where $h(x) = \sup \big ( \{ 0 \} \cup \{ |
		\perpproject{S} (\chi)-f( \project S ( \chi ) ) | \with \chi
		\in A_j \cap \project S^{-1} \lIm \{ x \} \rIm \} \big )$},
	\end{multline*}
	for $x \in S$.  The proof of \eqref{item:rademacher_sets:closeness}
	will be concluded by showing that for each positive integer $j$ the
	conclusion of \eqref{item:rademacher_sets:closeness} holds at $x+f(x)$
	for $\mathscr{H}^\vdim$ almost all $x \in D_j$.
	
	Evidently, the set $A_j$ is pointwise [strongly pointwise]
	differentiable of order~$\gamma$ at $y$ with $\pt \Der^i A_j (y,
	\cdot) = \pt \Der^i B(y,\cdot)$ for $i = 0, \ldots, k$ whenever $y \in
	Y$ and $j$ is a positive integer.  Next, it will be shown that
	\begin{equation*}
		A_j \cap \project S^{-1} \lIm \cball x{\lambda/j} \rIm \subset
		\big \{ \chi \with | \perpproject S ( \chi-a ) | \leq j |
		\project S ( \chi - a ) | \big \}
	\end{equation*}
	whenever $x \in D_j$, $a = x+f(x)$, and $j$ is a positive integer; in
	fact, if $\chi \in A_j$ and $| \project S ( \chi ) - x | \leq
	\lambda/j$, then defining $b = \project S ( \chi ) + f ( \project
	S ( \chi ) )$ one estimates
	\begin{equation*}
		|b-a| \leq ( 1 + \kappa ) | \project S (b-a) | \leq 1/(2j),
		\quad |\chi-a| \leq | \chi - b | + | b - a |  < 1/j,
	\end{equation*}
	hence $| \perpproject{S} ( \chi-a ) | \leq j | \project S ( \chi - a )
	|$ as $a \in C_j$.  Applying
	\ref{thm:eq_diff_sets}\,\eqref{item:eq_diff_sets:add_fct}
	$\big
	[$\ref{thm:eq_diff_sets}\,\eqref{item:eq_diff_sets:add_manifold}\,\eqref{item:eq_diff_sets:add_fct}$\big]$
	with $\kappa$, $A$ and $a$ replaced by $j$, $A_j \cap \project{S}^{-1}
	\lIm \cball x{\lambda/j} \rIm$ and $x+f(x)$ then yields
	\begin{equation*}
		\limsup_{r \to 0+} r^{-k-\alpha} \sup g_j \lIm \cball xr \rIm
		< \infty \quad \text{whenever $x \in D_j$ and $j = 1, 2, 3,
		\ldots$}
	\end{equation*}
	implying, by \ref{thm:big_O_little_o}, that
	\begin{equation*}
		\lim_{r \to 0+} r^{-k-\alpha} \sup g_j \lIm \cball xr \rIm = 0
		\quad \text{for $\mathscr{H}^\vdim$ almost all $x \in D_j$}
	\end{equation*}
	whenever $j$ is a positive integer.  Since
	\begin{gather*}
		\sup \dist ( \cdot, B ) \lIm A_j \cap \cball ar \rIm \leq \sup
		g_j \lIm \cball {\project{S} (a)} r \rIm \\
		\Big [ \sup \big ( \dist ( \cdot, B ) \lIm A_j \cap \cball ar
		\rIm \cup \dist ( \cdot, A_j ) \lIm B \cap \cball ar \rIm \big
		) \leq \sup g_j \lIm \cball {\project{S} (a)} r \rIm \Big ]
	\end{gather*}
	for $a \in C_j \cap Y$ and $0 < r < \infty$ and $A_j$ is a
	neighbourhood of $a$ relative to $A$ for such $a$, the conclusion of
	\eqref{item:rademacher_sets:closeness} now follows noting
	\ref{miniremark:diff_sets}.

	To prove \eqref{item:rademacher_sets:points}, consider $a \in B \cap
	X$ satisfying the conditions of \eqref{item:rademacher_sets:closeness}
	with respect to an $\vdim$ dimensional submanifold $B$ of class~$k+1$
	of $\rel^\adim$, as $\mathscr{H}^\vdim$ almost all $a \in X$ do by
	\eqref{item:rademacher_sets:approx},
	\cite[\hyperlink{3_1_15}{3.1.15}]{MR41:1976}, and
	\eqref{item:rademacher_sets:closeness}.  These conditions imply
	firstly that
	\begin{equation*}
		\lim_{r \to 0+} r^{-1} \sup | \dist (\cdot,A) - \dist (\cdot,
		B ) | \lIm \cball ar \rIm = 0
	\end{equation*}
	by \ref{miniremark:cone} and
	\ref{thm:eq_diff_sets}\,\eqref{item:eq_diff_sets:add_manifold} with
	$B$ replaced by $\Tan (A,a)$, and then that $A$ is pointwise [strongly
	pointwise] differentiable of order $k+1$ at $a$.
\end{proof}

\begin{remark}
	In the terminology of Anzellotti and Serapioni
	\eqref{item:rademacher_sets:approx}~states that $X$ is the union of a
	countable family whose members are
	``$(\mathscr{H}^\vdim,\vdim)$~rectifiable of
	class~$\mathscr{C}^{k,\alpha}$'', see \cite[3.1]{MR1285779}.
\end{remark}

\section{Approximate versus pointwise differentiation} \label{sect:example}

In this section it is shown in \ref{example:hesitate_to_vanish} that
approximate differentiability of every positive integer order does not entail
almost everywhere pointwise differentiability of order strictly larger than
$k$ for functions of class~$k$.  This will be used to contrast a result on
varifolds, see \ref{corollary:approx_implies_pt} and
\ref{remark:better_than_c2}.  The main lemma in this regard is
\ref{lemma:hesitate_to_vanish} which provides for every closed set and every
modulus of continuity a function of class~$k$ which realises the closed set as
its zero set and whose decay near that set is controlled from above and below
by the given modulus of continuity.

In this section, as in \cite[p.~993]{MR3528825}, each statement asserting
the existence of a number $\Gamma$ will give rise to a function depending on
the parameters determining it whose name is $\Gamma_{\text{x.y}}$, where x.y
denotes the number of the statement.

\begin{lemma} \label{lemma:cz_distance}
	Suppose $A$ is a closed subset of of $\rel^\vdim$, $\delta :
	\rel^\vdim \to \rel$ satisfies $\delta (x) = \dist (x,A)$ for $x \in
	\rel^\vdim$, $U = \{ x \with \delta (x) < 1 \}$, and $k$ is a positive
	integer.

	Then there exists a function $g : U \without A \to \rel$ of class
	$\infty$ such that
	\begin{gather*}
		\Gamma^{-1} \leq g(x)/\delta (x) \leq 1 \quad \text{and} \quad
		\big \| \Der^i g (x) \big \| \leq \Gamma \delta(x)^{1-i}
		\quad \text{for $i = 0, \ldots, k$}
	\end{gather*}
	whenever $x \in U \without A$, where $1 \leq \Gamma < \infty$ is
	determined by $k$ and $\vdim$.
\end{lemma}

\begin{proof}
	This is immediate from \cite[3.6.1]{MR1014685}.
\end{proof}

\begin{definition}
	$\omega$ is termed \emph{modulus of continuity} if and only if it is a
	function $\omega : \{ t \with 0 \leq t \leq 1 \} \to \{ t \with 0 \leq
	t \leq 1 \}$ satisfying
	\begin{gather*}
		\lim_{t \to 0+} \omega (t) = 0, \qquad \text{$\omega (t) = 0$
		if and only if $t=0$ whenever $0 \leq t \leq 1$}, \\
		\text{$\omega(s) \leq \omega(t)$ whenever $0 \leq s \leq t
		\leq 1$}.
	\end{gather*}
\end{definition}

\begin{lemma} \label{lemma:modulus_of_continuity}
	Suppose $\omega$ is a modulus of continuity.

	Then there exists a modulus of continuity $\psi$ such that $\psi | \{
	t \with t > 0 \}$ is of class~$\infty$ relative to $\{ t \with 0 < t
	\leq 1 \}$ and $\omega (t/4) \leq \psi (t) \leq \omega (t)$ for $0
	\leq t \leq 1$.%
	\begin{footnote}
		{If $A \subset \rel$ and $f : A \to \rel$ then $f$ is of class
		$\infty$ relative to $A$ if and only if there exist an open
		subset $U$ of $\rel$ and $g : U \to \rel$ of class $\infty$
		with $A \subset U$ and $f=g|A$, see \cite[3.1.22]{MR41:1976}.}
	\end{footnote}
\end{lemma}

\begin{proof}
	Constructing (for instance by means of a partition of unity) a modulus
	of continuity $\psi$ such that $\psi | \{ t \with t > 0 \}$ is of
	class $\infty$ relative to $\{ t \with 0 < t \leq 1 \}$ and $\psi
	(2^{-i}) = \omega ( 2^{-i-1} )$ whenever $i$ is a nonnegative integer,
	one readily verifies the conclusion.
\end{proof}

\begin{lemma} \label{lemma:hesitate_to_vanish}
	Suppose $A$ is a closed subset of $\rel^\vdim$, $\delta : \rel^\vdim
	\to \rel$ satisfies $\delta (x) = \dist (x,A)$ for $x \in\rel^\vdim$,
	$U = \{ x \with \delta (x) < 1 \}$, $\omega$ is a modulus of
	continuity, and $k$ is a positive integer.

	Then there exists $f : U \to \rel$ of class $k$ such that $f | U
	\without A$ is of class $\infty$ and
	\begin{gather*}
		f(x) \geq \Gamma^{-1} \omega ( \delta(x)/\Gamma ) \delta(x)^k
		\quad \text{for $x \in U \without A$}, \\
		\Der^i f (a) = 0 \quad \text{for $a \in A$}, \qquad
		\big \| \Der^i f(x) \big \| \leq \Gamma
		\omega ( \delta (x) ) \delta(x)^{k-i} \quad \text{for $x \in U
		\without A$}
	\end{gather*}
	whenever $i = 0, \ldots, k$, where $1 \leq \Gamma < \infty$ is
	determined by $k$ and $\vdim$.
\end{lemma}

\begin{proof}
	In view of \ref{lemma:modulus_of_continuity} the problem reduces to
	the case that $\omega | \{ t \with t > 0 \}$ is of class $\infty$
	relative to $\{ t \with 0 < t \leq 1 \}$.
	
	In this case there exists a function $h : \{ y \with y < 1 \} \to
	\rel$ of class $k$ such that%
	\begin{footnote}
		{If $g$ maps a subset of $\rel$ into $\rel$ and $g$ is $k$
		times differentiable at $x$, then $g^{(k)}(x) \in \rel$
		denotes the $k$-th derivative of $g$ at $x$, see
		\cite[\hyperlink{3_1_11}{3.1.11}]{MR41:1976}.}
	\end{footnote}
	\begin{gather*}
		h(y) = 0 \quad \text{for $- \infty < y \leq 0$}, \qquad
		h^{(k)} (y) = \omega (y) \quad \text{for $0 \leq y < 1$}.
	\end{gather*}
	Then \cite[\hyperlink{3_1_11}{3.1.11}]{MR41:1976} implies
	\begin{gather*}
		h^{(i)} (y) =  \big ( y^{k-i}/(k-i)! \big ) \tint 01
		(k-i)(1-t)^{k-i-1} \omega (ty) \ud \mathscr{L}^1 \, t
	\end{gather*}
	for $0 \leq y < 1$ and $i = 0, \ldots, k$.  Consequently, one obtains
	the estimates
	\begin{gather*}
		h(y) \geq (y^k/k!) \omega (y/2) \tint {1/2}1 k(1-t)^{k-1} \ud
		\mathscr{L}^1 \, t = 2^{-k} \omega (y/2) y^k/k!, \\
		\big | h^{(i)} (y) \big | \leq \omega (y) y^{k-i}/(k-i)!
		\quad \text{for $i = 0, \ldots, k$}
	\end{gather*}
	whenever $0 \leq y < 1$.

	Next, choose $g$ as in \ref{lemma:cz_distance}, abbreviate $\Delta =
	\Gamma_{\ref{lemma:cz_distance}} ( k, \vdim )$, and define $f : U \to
	\rel$ by $f(a) = 0$ for $a \in A$ and $f(x)=h(g(x))$ for $x \in U
	\without A$.  Defining $S(i)$ to be the set of all $k$ termed
	sequences $\alpha$ of nonnegative integers with $\sum_{j=1}^k j
	\alpha_j = i$, one estimates, using \cite[\hyperlink{1_10_5}{1.10.5},
	\hyperlink{3_1_11}{3.1.11}]{MR41:1976}, the estimates for $h$, and
	\ref{lemma:cz_distance},
	\begin{align*}
		\big \| \Der^i f (x) \big \| /i! & \leq \sum_{\alpha \in S(i)}
		\big | h^{(\sum \alpha)} (g(x)) \big | \prod_{j=1}^k \big ( \|
		\Der^j g(x) \|/j! \big )^{\alpha_j}/(\alpha_j)! \\
		& \leq \omega (g(x)) \sum_{\alpha \in S(i)} \big (
		g(x)^{k-\sum \alpha} / (k-\tsum{}{} \alpha)! \big ) \big (
		\Delta^{\sum \alpha} / \alpha! \big ) \delta (x)^{(\sum
		\alpha)-i} \\
		& \leq \Delta_i \omega (\delta(x)) \delta(x)^{k-i},
	\end{align*}
	for $x \in U \without A$ and $i = 0, \ldots, k$, where $\Delta_i =
	\sum_{\alpha \in S(i)} \big ( \Delta^{\sum \alpha}/\alpha!  \big )
	\big / \big (k-\sum \alpha \big )!$.  Inductively one infers
	that $f$ is $i$ times differentiable with $\Der^i f(a) = 0$ for $a \in
	A$ and $\Der^i f$ is continuous for $i = 0, \ldots, k$.  Since the
	estimates for $h$ and \ref{lemma:cz_distance} yield
	\begin{gather*}
		f(x) \geq 2^{-k} \omega ( g(x)/2 ) g(x)^k/k! \geq
		(2\Delta)^{-k} \omega ( \delta (x)/ (2\Delta)) \delta (x)^k/k!
		\quad \text{for $x \in U \without A$},
	\end{gather*}
	one may take $\Gamma = \sup \big ( \{ (2 \Delta)^k k!\} \cup \{
	\Delta_i i! \with i = 0, \ldots, k \} \big )$ in the present case.
\end{proof}

\begin{miniremark} [see for instance \protect{\cite[7.8]{MR3625810}}]
	\label{miniremark:planes}
	If $\vdim$ is a positive integer, then $\unitmeasure \vdim \leq 6$.%
	\begin{footnote}
		{By definition $\unitmeasure \vdim =
		\measureball{\mathscr{L}^\vdim}{\cball 01}$, see
		\cite[2.7.16\,(1)]{MR41:1976}.}
	\end{footnote}
\end{miniremark}

\begin{miniremark} [see \protect{\cite[5.1.9]{MR41:1976}}]
	\label{miniremark:pp_and_qq}
	Suppose $\vdim$ is a positive integer and $\adim = \vdim+1$.  Then
	$\Mypp : \rel^\adim \to \rel^\vdim$ and $\Myqq : \rel^\adim \to \rel$
	satisfy $\Mypp ( z ) = (z_1, \ldots, z_\vdim )$ and $\Myqq (z) =
	z_\adim$ for $z = (z_1, \ldots, z_\adim ) \in \rel^\adim$.
\end{miniremark}

\begin{example} \label{example:hesitate_to_vanish}
	Suppose $k$ and $\vdim$ are positive integers.

	Then there exist a closed subset $A$ of $\rel^\vdim$ and a nonnegative
	function $f : \rel^\vdim \to \rel$ of class $k$ satisfying
	\begin{gather*}
		\Der^i f(a) = 0 \quad \text{whenever $a \in A$ and $i = 0,
		\ldots, k$}, \qquad \mathscr{L}^\vdim (A) > 0, \\
		\text{$\limsup_{x \to a} |x-a|^{-k-\alpha} f(x) = \infty$ for
		$0 < \alpha \leq 1$} \quad \text{for $\mathscr{L}^\vdim$
		almost all $a \in A$}.
	\end{gather*}
	In particular, if $0 < \alpha \leq 1$, then $\rel^\adim \cap \{ z
	\with \Myqq (z) = f ( \Mypp (z) ) \}$ is not pointwise differentiable of
	order $(k,\alpha)$ at $c$ for $\mathscr{H}^\vdim$ almost all $c \in
	\Mypp^\ast \lIm A \rIm$, see \ref{miniremark:pp_and_qq}.%
	\begin{footnote}
		{The adjoint linear map $\Mypp^\ast : \rel^\vdim \to \rel^\adim$
		associated to $\Mypp$ satisfies $\Mypp^\ast (x) = (x_1, \ldots,
		x_\vdim, 0 ) \in\rel^\adim$ for $x = (x_1, \ldots, x_\vdim )
		\in \rel^\vdim$, see \cite[1.7.4]{MR41:1976}.}
	\end{footnote}
\end{example}

\begin{proof} [Construction]
	Evidently, it is sufficient to prove the assertion with ``$f :
	\rel^\vdim \to \rel$'' replaced by ``$f : \{ x \with \dist (x,A) < 1 \}
	\to \rel$''.  Define a modulus of continuity $\omega$ by
	\begin{gather*}
		\omega (0) = 0, \qquad \omega (t) = ( 1+ \log (1/t) )^{-1}
		\quad \text{for $0 < t \leq 1$}.
	\end{gather*}
	By Kolasiński and the author \cite[2.5]{MR3625810}, there
	exist $B \subset \{ r \with r > 0 \}$, a Borel subset $X$ of
	$\rel^\vdim$, and a collection $F$ of open balls in $\rel^\vdim$ such
	that
	\begin{gather*}
		\inf B = 0, \quad \mathscr{L}^\vdim (X) > 0, \quad
		\mathscr{L}^\vdim ( ( \Clos X) \without X ) = 0,
	\end{gather*}
	and such that for $a \in X$ and $r \in B$
	there exists $U \in F$ with
	\begin{gather*}
		U \subset \oball ar \without X \quad \text{and} \quad
		\mathscr{L}^\vdim (U) \geq \omega (r) r^\vdim.
	\end{gather*}
	Define $A = \Clos X$, take $\delta$ and $f$ as in
	\ref{lemma:hesitate_to_vanish}, and let $\Delta = 6^k
	\Gamma_{\ref{lemma:hesitate_to_vanish}} (k,\vdim)$.  If $a \in
	X$, $r \in B$, and $U = \oball xs \in F$ are related as above then one
	estimates, using \ref{miniremark:planes},
	\begin{gather*}
		|x-a| \leq r, \quad \delta(x) \geq s \geq \omega (r)^{1/\vdim}
		\unitmeasure \vdim^{-1/\vdim} r \geq \omega (|x-a|)^{1/\vdim}
		|x-a|/6, \\
		f(x) \geq \Delta^{-1} \omega \big ( \omega (|x-a|)^{1/\vdim}
		|x-a|/\Delta \big ) \omega ( |x-a| )^{k/\vdim} |x-a|^k.
	\end{gather*}
	The principal conclusion now follows from the definition of $\omega$
	and the postscript is a consequence of \ref{corollary:cmp_fct_set}.
\end{proof}

\section{Intersecting a stationary varifold with a plane}
\label{sect:varifolds}

In this section a regularity property of the support of the weight measure of
a stationary integral varifold is proven in
\ref{thm:pointwise_approx_pt_diff}.  Namely, if near one of its points it is
contained in the union of a plane of the same dimension as the varifold and a
set with density zero at that point then it is strongly pointwise
differentiable of every positive integer order at that point.  This applies to
almost all points of the intersection with such a plane, see
\ref{corollary:approx_implies_pt}.  The differentiability condition obtained
not only encodes a vanishing phenomenon of infinite order but additionally
places a strong restriction on the size of ``holes'' the varifold may have
near that point both with respect to distance and Hausdorff measure, see
\ref{miniremark:cone},
\ref{thm:eq_diff_sets}\,\eqref{item:eq_diff_sets:add_fct}, and Kolasiński
and the author \cite[10.4, 11.7, 11.8]{MR3625810}.

\begin{lemma} \label{lemma:zero}
	Suppose $\vdim$ and $\adim$ are positive integers, $2 \leq \vdim \leq
	\adim$, $a \in \rel^\adim$, $0 < r < \infty$, $V \in \Var_\vdim (
	\oball ar )$, $\delta V =0$, $\density^\vdim ( \| V \|, x ) \geq 1$
	for $\| V \|$ almost all $x$, $T \in \grass \adim \vdim$, $0 < l
	< \infty$,
	\begin{gather*}
		\| V \| ( \oball as \without \{ x \with x-a \in T \} ) \leq
		2^{-\vdim-l} \unitmeasure \vdim s^\vdim \quad \text{for
		$0 < s \leq r$},
	\end{gather*}
	and $\psi (s) = \sup \{ \dist (x-a,T) \with x \in \oball as \cap \spt
	\| V \| \}$ for $0 < s \leq r$.

	Then there holds
	\begin{gather*}
		\psi (s) \leq 2^l (s/r)^l \psi (r) \quad \text{for
		$0 < s \leq r$}.
	\end{gather*}
\end{lemma}

\begin{proof}
	If $f : \oball ar \to \rel$ is a convex function, then
	\begin{gather*}
		\tint{}{} f (x) D ( \grad \zeta ) (x) \bullet \project S
		\ud V\,(x,S) \geq 0 \quad \text{for $0 \leq \zeta \in
		\mathscr{D} ( \oball ar, \rel )$};
	\end{gather*}
	in fact, convolution reduces the problem to the case that $f$ is of
	class $\infty$ which follows from Allard
	\cite[7.5\,(1)\,(2)]{MR0307015}. Take $f(x) = \dist (x-a,T)$ for $x
	\in \oball ar$, abbreviating $\phi (s) = s^{-\vdim} \tint{\oball as}{}
	f \ud \| V \|$ for $0 < s \leq r$, and denote by $e_1, \ldots,
	e_\adim$ the standard orthonormal base of $\rel^\adim$.  Applying
	Michael and Simon \cite[3.4]{MR0344978} with $M$, $U$, $\mu$, $\tilde
	g^{ij}(x)$, $\mathscr{H}_i(x)$, $\chi$, $\Lambda$, $\xi$, and
	$\varrho$ replaced by $\oball ar$, $\oball ar$, $\| V \|$,
	$\project{\Tan^\vdim ( \| V \|, x)}(e_i) \bullet e_j$, $0$, $f$, $0$,
	$x$, and $s/2$ for $x \in \oball a{s/2} \cap \spt \| V \|$ implies
	\begin{gather*}
		\psi (s/2) \leq \unitmeasure{\vdim}^{-1} 2^\vdim \phi (s)
		\quad \text{for $0 < s \leq r$}.
	\end{gather*}
	Moreover, Hölder's inequality yields
	\begin{gather*}
		\phi (s) \leq 2^{-\vdim-l} \unitmeasure \vdim \psi (s)
		\quad \text{for $0 < s \leq r$}.
	\end{gather*}
	Together one obtains the conclusion; in fact, it is evident if $s \geq
	r/2$ and if it holds for some $0 < s \leq r$ then
	\begin{gather*}
		\psi (s/2) \leq 2^{-l} \psi (s) \leq (s/r)^l \psi
		(r)
	\end{gather*}
	and the conclusion holds for $s/2$.
\end{proof}

\begin{theorem} \label{thm:pointwise_approx_pt_diff}
	Suppose $\vdim$ and $\adim$ are positive integers, $2 \leq \vdim \leq
	\adim$, $U$ is an open subset of $\rel^\adim$, $V \in \IVar_\vdim ( U
	)$, $\delta V = 0$, and
	\begin{gather*}
		a \in \spt \| V \|, \quad T \in \grass \adim \vdim, \quad
		\density^\vdim ( \| V \| \restrict U \cap \{ x \with x-a
		\notin T \}, a ) = 0.
	\end{gather*}

	Then $\spt \| V \|$ is strongly pointwise differentiable of every
	positive integer order at $a$, $\Tan ( \spt \| V \|, a) = T$, and
	\begin{equation*}
		\lim_{s \to 0+} s^{-l} \tint{\oball as \times \grass \adim
		\vdim}{} \| \project S - \project T \|^2 \ud V \, (x,S) = 0
		\quad \text{for $0 < l < \infty$}.
	\end{equation*}
\end{theorem}

\begin{proof}
	Assume $a = 0$ and abbreviate $A = \spt \| V \|$.  Then
	\ref{lemma:zero} yields
	\begin{gather*}
		\lim_{s \to 0+} s^{-l} \sup \dist ( \cdot, T ) \lIm A \cap
		\cball 0s \rIm = 0 \quad \text{for $0 < l < \infty$},
	\end{gather*}
	in particular one may assume that $A \subset \big \{ \chi \with
	|\perpproject T (\chi) | \leq | \project T ( \chi ) | \big \}$.
	Notice that $1 \leq \density^\vdim ( \| V \|, 0 ) < \infty$ by Allard
	\cite[5.1\,(2), 8.6]{MR0307015}.  One infers
	\begin{gather*}
		\lim_{s \to 0+} s^{-\vdim-2(l-1)} \tint{\oball 0s \times
		\grass \adim \vdim}{} \| \project S - \project T \|^2 \ud
		V\,(x,S) = 0 \quad \text{for $0 < l < \infty$}
	\end{gather*}
	from Allard \cite[8.13]{MR0307015}.  This implies
	\begin{equation*}
		\lim_{s \to 0+} s^{-\vdim} \tint{}{} f (s^{-1}x,S) \ud V \,
		(x,S) = \density^\vdim ( \| V \|, 0 ) \tint T{} f(x,T) \ud
		\mathscr{H}^\vdim \, x
	\end{equation*}
	whenever $f : \rel^\adim \times \grass \adim \vdim \to \rel$ is a
	continuous function with compact support and that $\density^\vdim ( \|
	V \|, 0 )$ is an integer by Allard \cite[3.4, 4.6\,(3),
	6.4]{MR0307015}, hence one observes that
	\begin{gather*}
		\lim_{s \to 0+} s^{-\vdim-2(l-1)} \mathscr{H}^\vdim ( T \cap
		\cball 0s \without \project{T} \lIm A \rIm ) = 0, \\
		\lim_{s \to 0+} s^{-1-2(l-1)/\vdim} \sup \dist ( \cdot,
		\project{T} \lIm A \rIm ) \lIm T \cap \cball 0s \rIm = 0
	\end{gather*}
	for $0 < l < \infty$ by Kolasiński and the author
	\cite[10.4]{MR3625810}.  Therefore $\project T \lIm A \rIm$ is
	strongly pointwise differentiable of every positive integer order
	at~$0$ and $\Tan ( \project T \lIm A \rIm, 0 ) = T$ by
	\ref{miniremark:diff_sets} and \ref{remark:same_tangent_cone}.  Now,
	the conclusion follows from
	\ref{thm:eq_diff_sets}\,\eqref{item:eq_diff_sets:add_half_manifold}
	with $S$ and $B$ replaced by $T$ and $T$.
\end{proof}

\begin{corollary} \label{corollary:approx_implies_pt}
	Suppose $\vdim$ and $\adim$ are positive integers, $2 \leq \vdim \leq
	\adim$, $U$ is an open subset of $\rel^\adim$, $V \in \IVar_\vdim
	(U)$, $\delta V = 0$, and $T \in \grass \adim \vdim$.

	Then $\spt \| V \|$ is strongly pointwise differentiable of every
	positive integer order at $a$, $\Tan ( \spt \| V \|, a ) = T$, and
	\begin{equation*}
		\lim_{s \to 0+} s^{-l} \tint{\oball as \times \grass \adim
		\vdim}{} \| \project S - \project T \|^2 \ud V \, (x,S) = 0
		\quad \text{for $0 < l < \infty$}
	\end{equation*}
	for $\mathscr{H}^\vdim$ almost all $a \in T \cap U$.
\end{corollary}

\begin{proof}
	This is a consequence of
	\cite[\hyperlink{2_10_19}{2.10.19}\,(4)]{MR41:1976} and
	\ref{thm:pointwise_approx_pt_diff}.
\end{proof}

\begin{remark} \label{remark:better_than_c2}
	The behaviour of $\spt \| V \|$ exhibited in
	\ref{corollary:approx_implies_pt} is not shared by all closed $\vdim$
	dimensional submanifolds of $\rel^{\vdim+1}$ of class $2$ by
	\ref{example:hesitate_to_vanish} with $k=2$.
\end{remark}

\begin{remark}
	In case $\vdim = 1$ a complete description of the structure
	$\mathscr{H}^1$ almost everywhere of $\spt \| V \|$ was obtained by
	Allard and Almgren in \cite[p.~89]{MR0425741}.
\end{remark}

\appendix

\section{Items employed from Federer's treatise} \label{sect:appendix}

For the convenience of the reader, Table~\ref{table:fed69} provides a brief
list of the results employed from \cite{MR41:1976}.  Items which merely
provide background are not listed.

\begin{table}[htb]
	\begin{tabular}{|lp{285pt}|}
		\hline
		Number & Description \\
		\hline
		\raisebox{\ht\strutbox}{\hypertarget{1_10_2}}%
		{1.10.2} &  Algebra of symmetric forms. \\
		\raisebox{\ht\strutbox}{\hypertarget{1_10_4}}%
		{1.10.4} & Polynomial functions and Taylor's formula.
		\\
		\raisebox{\ht\strutbox}{\hypertarget{1_10_5}}%
		{1.10.5} & Estimates of seminorms related to symmetric
		forms. \\
		\raisebox{\ht\strutbox}{\hypertarget{2_2_7}}%
		{2.2.7} & Basic properties of Lipschitzian maps. \\
		\raisebox{\ht\strutbox}{\hypertarget{2_2_10}}%
		{2.2.10} & Mapping properties of Borel and Suslin
		sets, see also \cite[15.1]{MR1321597}. \\
		\raisebox{\ht\strutbox}{\hypertarget{2_7_1}}%
		{2.7.1} & Transitive left actions on homogeneous
		spaces.  \\
		\raisebox{\ht\strutbox}{\hypertarget{2_9_17}}%
		{2.9.17} & A differentiation theorem for general measures. \\
		\raisebox{\ht\strutbox}{\hypertarget{2_10_19}}%
		{2.10.19} & Properties of densities. \\
		\raisebox{\ht\strutbox}{\hypertarget{2_10_21}}%
		{2.10.21} & Includes Ascoli theorem for Lipschitzian
		functions.  \\
		\raisebox{\ht\strutbox}{\hypertarget{3_1_1}}%
		{3.1.1} & First order differentials. \\
		\raisebox{\ht\strutbox}{\hypertarget{3_1_11}}%
		{3.1.11} & Higher differentials: Taylor formula,
		$k$~jets, composition formula. \\
		\raisebox{\ht\strutbox}{\hypertarget{3_1_14}}%
		{3.1.14} & Whitney's extension theorem. \\
		\raisebox{\ht\strutbox}{\hypertarget{3_1_15}}%
		{3.1.15} & Lusin type approximation of functions of class
		$(k,1)$ by functions of class $k+1$, see also Whitney
		\cite[Theorem~4]{MR0043878}. \\
		\raisebox{\ht\strutbox}{\hypertarget{3_1_18}}%
		{3.1.18} & Consequences of the inverse function
		theorem. \\
		\raisebox{\ht\strutbox}{\hypertarget{3_1_19}}%
		{3.1.19} & Characterisations of submanifolds of
		Euclidean space. \\
		\raisebox{\ht\strutbox}{\hypertarget{3_1_21}}%
		{3.1.21} & Tangent cones and their mapping properties.
		\\
		\raisebox{\ht\strutbox}{\hypertarget{3_2_28}}%
		{3.2.28} & Includes Grassmann manifolds treated as homogeneous
		spaces.  \\
		\raisebox{\ht\strutbox}{\hypertarget{3_2_29}}%
		{3.2.29} & Characterising countably
		$(\mathscr{H}^\vdim,\vdim)$ rectifiable sets by
		coverings consisting of submanifolds of class~$1$. \\
		\raisebox{\ht\strutbox}{\hypertarget{3_3_5}}%
		{3.3.5} & A basic rectifiability lemma. \\
		\hline
	\end{tabular}
	\caption{Items employed from \protect{\cite{MR41:1976}}.}
	\label{table:fed69}
\end{table}

\medskip \noindent \textsc{Affiliations}

\medskip \noindent Institute of Mathematics, University of Leipzig \newline
Augustusplatz 10, 04109 \textsc{Leipzig}, \textsc{Germany} \smallskip \newline
Max Planck Institute for Mathematics in the Sciences \newline Inselstraße 22,
04103 \textsc{Leipzig}, \textsc{Germany}

\medskip \noindent \textsc{Email addresses}

\medskip \noindent
\href{mailto:Ulrich.Menne@math.uni-leipzig.de}{Ulrich.Menne@math.uni-leipzig.de}
\quad \href{mailto:Ulrich.Menne@mis.mpg.de}{Ulrich.Menne@mis.mpg.de}

\end{document}